\documentclass{amsart}
\usepackage[margin=1.5in]{geometry}
\def\draftdate{\today}

\usepackage{amsmath}
\usepackage{amsfonts}
\usepackage{amssymb}
\usepackage{amsthm}
\usepackage{comment}
\usepackage{epsfig}
\usepackage{psfrag}
\usepackage{mathrsfs}
\usepackage{amscd}
\usepackage[all]{xy}
\usepackage{rotating}
\usepackage{lscape}
\usepackage{amsbsy}
\usepackage{verbatim}
\usepackage{moreverb}
\usepackage[colorlinks=true, linkcolor=blue, citecolor=blue]{hyperref}
\usepackage{color}
\usepackage[titletoc]{appendix} 
\usepackage{enumitem}
\usepackage{scalefnt}
\usepackage{graphicx}
\usepackage{comment}
\usepackage{aliascnt}

\makeatletter
\newtheorem*{rep@theorem}{\rep@title}
\newcommand{\newreptheorem}[2]{%
\newenvironment{rep#1}[1]{%
 \def\rep@title{#2 \ref{##1}}%
 \begin{rep@theorem}}%
 {\end{rep@theorem}}}
\makeatother

\newtheorem{thm}{Theorem}[section]
\newreptheorem{thm}{Theorem}
\newreptheorem{prop}{Proposition}
\newreptheorem{lem}{Lemma}
\newreptheorem{defn}{Definition}

\newaliascnt{lem}{thm}
\newaliascnt{prop}{thm}
\newaliascnt{cor}{thm}

\newtheorem{prop}[prop]{Proposition}
\newtheorem{cor}[cor]{Corollary}
\newreptheorem{cor}{Corollary}
\aliascntresetthe{lem}
\aliascntresetthe{prop}
\aliascntresetthe{cor}

\theoremstyle{definition}
\newtheorem{defn}[thm]{Definition}
\theoremstyle{definition}
\newtheorem{ex}[thm]{Example}
\theoremstyle{definition}
\newtheorem{rem}[thm]{Remark}
\theoremstyle{definition}

\theoremstyle{definition}
\newtheorem{nota}[thm]{Notation}

\newcommand{\oo}{\infty}

\newcommand{\la}{\langle}
\newcommand{\ra}{\rangle}

\mathchardef\dash="2D

\renewcommand{\L}{\mathrm{L}}

\DeclareMathOperator{\Ind}{Ind}
\DeclareMathOperator{\Hom}{Hom}
\DeclareMathOperator{\cell}{cell}
\DeclareMathOperator{\Perf}{Perf}
\DeclareMathOperator{\Mod}{Mod}
\DeclareMathOperator{\Fun}{Fun}
\DeclareMathOperator{\idemtimes}{\widehat{\otimes}}
\DeclareMathOperator{\Stab}{Stab}
\DeclareMathOperator{\Idem}{Idem}

\DeclareMathOperator{\Id}{Id}
\DeclareMathOperator{\fincol}{f. c.}
\DeclareMathOperator{\Res}{Res}
\DeclareMathOperator{\LKan}{LKan}

\def\cA{\mathcal A}\def\cB{\mathcal B}\def\cC{\mathcal C}\def\cD{\mathcal D}
\def\cE{\mathcal E}\def\cF{\mathcal F}\def\cG{\mathcal G}
\def\cL{\mathcal L}
\def\cM{\mathcal M}\def\cN{\mathcal N}\def\cO{\mathcal O}\def\cP{\mathcal P}
\def\cS{\mathcal S}\def\cT{\mathcal T}

\def\SS{\mathbb S}

\begin{document}
\scalefont{1.1}

\title{Differential Graded Categories are k-linear Stable $\infty$-Categories}
\author{Lee Cohn}
\address{Department of Mathematics, The University of Texas, Austin, TX \ 78712}
\email{lcohn@math.utexas.edu}
\date{\draftdate}

\begin{abstract}
\scalefont{1.1}
We describe a comparison between pretriangulated differential graded categories and certain stable $\infty$-categories. Specifically, we use a model category structure on small differential graded categories over $k$ (a commutative ring with unit) where the weak equivalences are the Morita equivalences, and where the fibrant objects are in particular pretriangulated differential graded categories. We show the underlying $\oo$-category of this model category is equivalent to the $\oo$-category of small idempotent-complete $k$-linear stable $\oo$-categories.    
\end{abstract}

\maketitle

\centerline{Contents}
\renewcommand\contentsname{}
\tableofcontents

\section{Introduction}
In this paper we prove the folklore theorem (e.g. implicitly assumed in works such as \cite{BFN}, \cite{G}) that pretriangulated differential graded categories over $k$ (a commutative ring with unit) are $k$-linear stable $\oo$-categories. The basic setting for our work is the theory of $\oo$-categories (particularly stable $\oo$-categories), which provide a tractable way to handle  ``homotopical categories of homotopical categories" as well as homotopically meaningful categories of homotopical functors.  Thus, our comparison between pretriangulated dg categories and stable $\oo$-categories allows for an interpretation of a  ``homotopical category of homotopical, algebraic categories."  

Triangulated categories arise naturally in geometry (e.g. as derived categories or as stable homotopy categories).  However, triangulated categories lack homotopical properties such as functorial mapping cones.  Differential graded categories and stable $\oo$-categories are enhancements of triangulated categories that have such homotopical properties.  Differential graded categories are advantageous for triangulated enhancements because one can often perform explicit computations in this setting.  Stable $\oo$-categories are an alternative enhancement of triangulated categories that are more general than dg categories and are amenable to universal properties. 

In \cite{To1} and \cite{Tab2}, T\"{o}en and Tabuada construct a model category structure on the category of differential graded categories, $Cat^k_{dg}$, in which the weak equivalences are the Dwyer-Kan (DK)-equivalences.  That is, the weak equivalences are dg functors that induce both quasi-isomorphisms on the hom complexes and an ordinary categorical equivalence of homotopy categories. Following \cite{TabM}, we also use a ``Morita" model structure on dg categories, where the weak equivalences are now the larger class of dg functors that induce DK-equivalences on their categories of modules.  In this model category, fibrant dg categories are, in particular, pretriangulated.    

One of the impediments to working with these model categories is that while there is a monoidal structure on the category of dg categories given by the pointwise tensor product of dg categories, this tensor product does not respect either model category structure. Thus, dg categories do not form a monoidal model category for the pointwise tensor structure.  There are examples of two cofibrant dg categories whose pointwise tensor product is not cofibrant (with respect to both model category structures).  This is one reason why our comparison between dg categories and $k$-linear stable $\oo$-categories is not completely trivial.  To resolve this issue, we use ideas inspired by \cite{To1}, \cite{CT}, and \cite{BGT2} to derive the tensor product of dg categories via ``flat" dg categories. 

Following \cite{Tab1}, we use a Quillen equivalence between dg categories over $k$ (with the Morita model structure) and categories enriched in $Hk$-module symmetric spectra (also with the Morita model structure) to reformulate our questions in the language of spectral categories. Here, an $Hk$-module spectrum is a spectrum $A$, with an action map $Hk \wedge A \to A$ that satisfies the usual associativity and identity axioms. This Quillen equivalence may be interpreted as an enriched Dold-Kan correspondence because the mapping complexes in dg categories are chain complexes while the mapping complexes in spectral categories may be modeled using simplicial sets. 
 
Given an $E_{\oo}$-ring $R$ (e.g. $Hk$), $R$ may be considered as a homotopy coherent commutative algebra object in the $\oo$-category of spectra, $Sp$, and also as a homotopy coherent commutative algebra object in the model category of symmetric spectra $\cS$. As an object in $\cS$, we may form its category of perfect modules $\Perf(R)$, and restrict to the subcategory of perfect cell R-modules, $\Perf(R)^{\cell}$.  A perfect $R$-module is cell if it can be iteratively obtained from homotopy cofibers of maps whose sources are wedges of shifts of $R$. Let $Cat_{\cS}$ denote the category of spectral categories. Then a module over $\Perf(R)^{\cell}$ is a spectral category $\cC$ with an action map $\Perf(R)^{\cell} \wedge \cC \to \cC$ that preserves finite colimits and tensors with finite spectra in each variable and is equipped with the usual associativity maps. We let $\Mod_{\Perf(R)^{\cell}}(Cat_{\cS})$ denote the spectral category of modules over $\Perf(R)^{\cell}$.

Let $Cat_{\cS}$ denote the category of spectral categories, $Cat_{\cS}^{flat}$ the category of flat spectral categories \ref{flatsc}, $W'$ the collection of Morita equivalences in $\Mod_{R \dash \Mod^{\cell}}(Cat_{\cS})$, and $W$ the collection of Morita equivalences in $Cat_{\cS}$. We also let the superscript ``$\otimes$" denote the extra structure of a symmetric monoidal $\oo$-category.

Using the notation above, there is a localization of $\oo$-categories (see \ref{modeloo})
$$\theta: N(Cat_{\cS}^{flat})^{\otimes} \to N(Cat_{\cS}^{flat})[W^{-1}]^{\otimes}.$$ 
Since the nerve is monoidal \ref{nerve}, the map $\theta$ descends to a well-defined functor 
$$\theta: N(\Mod_{\Perf(R)^{\cell}}(Cat_{\cS})^{flat})[W'^{-1}] \to \Mod_{\Perf(R)}(N(Cat_{\cS}^{flat})[W^{-1}]^{\otimes})$$
because the nerve of a module category over $\Perf(R)^{\cell}$ in $Cat_{\cS}$ is naturally a module category over the $\oo$-category $\Perf(R)$.  
Our main theorem is: 

\begin{repthm}{main}
The functor 
$$\theta: N(\Mod_{\Perf(R)^{\cell}}(Cat_{\cS})^{flat})[W'^{-1}] \simeq \Mod_{\Perf(R)}(N(Cat_{\cS}^{flat})[W^{-1}]^{\otimes})$$
is an equivalence of $\oo$-categories.
\end{repthm}

In other words, this theorem asserts an equivalence between module categories over the spectral category $\Perf(R)^{\cell}$ and module categories over the stable $\oo$-category $\Perf(R)$ up to Morita equivalence. 

Next, we use an equivalence between module categories over $\Perf(R)^{\cell}$ with all finite colimits and closed under tensors with finite spectra and categories enriched in $R$-modules with all finite colimits and closed under tensors with finite spectra to prove the following. 
\begin{repprop}{newcomparison}
There is an equivalence of $\oo$-categories
$$N(Cat_{R \dash \Mod})[W^{-1}] \simeq N(\Mod_{\Perf(R)^{\cell}}(Cat_{\cS}))[W'^{-1}].$$
\end{repprop}

Using our main theorem above we deduce the following corollary.
\begin{repcor}{cor1} 
There is an equivalence of $\oo$-categories
$$ N(Cat_{R \dash \Mod})[W^{-1}] \simeq \Mod_{\Perf(R)}(N(Cat_{\cS}^{flat})[W^{-1}]^{\otimes}).$$
\end{repcor}

In other words, our corollary states that the underlying $\oo$-category of the model category of categories enriched in $R$-module spectra localized at the Morita equivalences and the $\oo$-category of module categories over the $\oo$-category $\Perf(R)$ in $N(Cat_{\cS}^{flat})[W^{-1}]$ are equivalent.  By \cite[3.8]{BGT2}, there is an equivalence
$$N(Cat_{\cS}^{flat})[W^{-1}]^{\otimes} \simeq (Cat^{perf}_{\oo})^{\otimes},$$
where $Cat^{perf}_{\oo}$ is the $\oo$-category of small stable idempotent-complete $\oo$-categories with exact functors.  
Thus, 
$$N(Cat_{R \dash \Mod})[W^{-1}] \simeq \Mod_{\Perf(R)}((Cat^{perf}_{\oo})^{\otimes}).$$ 

Applying this theorem to the Eilenberg-MacLane ring spectrum $Hk$ yields the following corollary.
\begin{repcor}{cor2}
There is an equivalence of $\oo$-categories
$$ N(Cat_{Hk \dash \Mod})[W^{-1}] \simeq Mod_{\Perf(Hk)}((Cat^{perf}_{\oo})^{\otimes}).$$
\end{repcor}

We now combine this result with the Quillen equivalence \cite{Tab1} 
$$Cat_{Hk \dash \Mod} \simeq Cat^k_{dg}$$ 
and \ref{pretridgsc} to obtain our main corollary.

\begin{repcor}{cor3} 
There is an equivalence of $\oo$-categories
$$ N(Cat^k_{dg})[W^{-1}] \simeq Mod_{\Perf(Hk)}((Cat^{perf}_{\oo})^{\otimes}).$$ That is, there is an equivalence between the underlying $\oo$-category of the model category of dg categories over $k$ localized at the Morita equivalences and the $\oo$-category of small idempotent-complete $k$-linear stable $\oo$-categories. 
\end{repcor}

Using the $\Ind$-construction we have a symmetric monoidal equivalence between $Cat^{perf}_{\oo}$ and compactly-generated presentable stable $\oo$-categories with functors that preserve colimits and compact objects (denoted $\cP r^{L}_{st, \omega}$).  This implies the following proposition.
\begin{repprop}{smalltolarge}
There is an equivalence of $\oo$-categories
$$\Ind: \Mod_{\Perf(Hk)}((Cat^{perf}_{\oo})^{\otimes}) \simeq \Mod_{Hk \dash \Mod}((\cP r^{L}_{st, \omega})^{\otimes})$$
\end{repprop}

We now have our second main corollary.
\begin{repcor}{cor4}
There is an equivalence of $\oo$-categories
$$ N(Cat^k_{dg})[W^{-1}] \simeq \Mod_{Hk \dash \Mod}((\cP r^{L}_{st, \omega})^{\otimes}).$$ In other words, the underlying $\oo$-category of the Morita model category structure on $Cat^k_{dg}$ is equivalent to the $\oo$-category of compactly-generated presentable $k$-linear stable $\oo$-categories with functors that preserve colimits and compact objects. 
\end{repcor}

Our main corollary is, first and foremost, a rectification result.  That is, given an idempotent-complete $k$-linear small stable $\oo$-category $\cC$, there exists a  pretriangulated dg category corresponding to $\cC$.  Thus, the category of dg categories localized at the Morita equivalences can be interpreted as a model for the $\oo$-category of small idempotent-complete $k$-linear stable $\oo$-categories. 

Acknowledgments: The author wishes to thank Andrew Blumberg for all of his help and guidance. The author also thanks  David Ben-Zvi, Owen Gwilliam, Aaron Royer, and Pavel Safronov for helpful discussions and/or comments.

\section{Differential Graded Categories}
\subsection{Review of Differential Graded Categories} \mbox{}\\
In this section we review the definition of a differential graded category, give several basic examples, and recall the model category structure on the category of dg categories given by T\"{o}en and Tabuada. The reader is referred to \cite{K} or \cite{To2} for a general introduction to dg categories.

We now fix $k$ any commutative ring with unit. Let $Ch(k)$ denote the category of chain complexes of $k$-modules.
\begin{defn}
A differential graded (dg) category $\cT$ over $k$ consists of the following data.
\begin{itemize}
\item A class of objects $Ob(\cT)$.
\item For every pair of objects $x,y \in \cT$ a complex $\cT(x,y) \in Ch(k)$.
\item For every triple $x,y,z \in \cT$ a composition $\cT(y,z) \otimes \cT(x,y) \to \cT(x,z)$.
\item A unit for every $x \in \cT$, $e_x: k \to \cT(x,x)$.
\\
These data are required to satisfy the usual associativity and unit conditions. 
\end{itemize}
\end{defn}

We say that a dg category is small if its class of objects $Ob(\cT)$ forms a set. 
\begin{nota}
We will denote the category of small dg categories by $Cat^k_{dg}$.
\end{nota}

The definition of a dg functor is the usual definition of an enriched functor.
\begin{defn}
Let $\cT$ and $\cT'$ be dg categories. A dg functor $\cF: \cT \to \cT'$ consists of:
\begin{itemize}
\item A map on objects $\cF: Ob(\cT) \to Ob(\cT')$.
\item A map of chain complexes $\cF_{x,y}: \cT(x,y) \to \cT'(\cF (x),\cF (y))$ for all $x,y \in Ob(\cT)$,
\end{itemize}
which is compatible with the units and compositions in the obvious sense.
\end{defn}

Every dg category has an underlying ordinary category given by restricting to the $0^{th}$ cohomology group of each hom complex.
\begin{defn}
Let $\cT$ be a dg category. The homotopy category of $\cT$, denoted by $[\cT]$, is the ordinary category given by:
\begin{itemize}
\item The objects of $[\cT]$, $Ob([\cT])$, are the objects of $\cT$.
\item For every pair of objects $x,y \in Ob([\cT])$, the set of morphisms $[\cT](x,y) := H^0(\cT(x,y))$.
\end{itemize}
\end{defn}

\begin{ex} A dg algebra is a dg category with one object.
\end{ex}

\begin{ex}
Chain complexes $Ch(k)$ form a dg category. Given two chain complexes, one can form a chain complex of morphisms between them. 
\end{ex}

A fact that we will implicitly use in the sequel is
\begin{ex}
If $\cC$ is a dg category, then so is its opposite category $\cC^{op}$.
\end{ex}

In \cite[2.2.1]{Tab2}, Tabuada shows that the category of dg categories over $k$ has a model category structure that we recall here. 
 
\begin{thm} The category of small dg categories, $Cat_{dg}^k$, has the structure of a combinatorial model category where 
\begin{itemize}
\item (W) The weak equivalences are the DK-equivalences.  That is, a dg functor $\cF: \cT \to \cT'$ is a weak equivalence if for any two objects $x,y \in \cT$ the morphism 
$\cT(x,y) \to \cT'(\cF(x),\cF(y))$ is a quasi-isomorphism of chain complexes and the induced functor $[\cF]: [\cT] \to [\cT']$ is an equivalence of categories.

\item(F) The fibrations are the DK-fibrations. That is, a dg functor $\cF: \cT \to \cT'$ is a fibration if for any two objects $x,y \in \cT$ the morphism 
$\cT(x,y) \to \cT'(\cF(x),\cF(y))$ is a fibration of chain complexes (i.e. is surjective in each degree), and for any isomorphism $u':x' \to y' \in [\cT']$ and any $y \in [\cT]$ such that $\cF(y)=y'$, there is an isomorphism $u: x \to y \in [\cT]$ such that $[\cF](u)=u'$.
\end{itemize}
\end{thm}

We will refer to this model structure as the Dwyer-Kan (or DK)-model structure on $Cat_{dg}^k$.

\begin{rem}
It is easy to see that $\cF: \cT \to \cT'$ is a DK-equivalence if and only if it induces an equivalence of graded homotopy categories $H^*(\cF): H^*(\cT) \to H^*(\cT')$.
\end{rem}

There is a pointwise tensor product of two dg categories inducing a monoidal structure on the category of dg categories.
\begin{defn} Let $\cC$ and $\cD$ be two dg categories. The pointwise tensor product $\cC \otimes \cD$ is the dg category specified by
\begin{itemize}
\item The objects of $\cC \otimes \cD$, $Ob( \cC \otimes \cD )$, are given by pairs of objects $(c,d)$, where $c \in \cC$ and $d \in \cD$.
\item  For every pair of objects $(c,d) , (c',d') \in \cC \otimes \cD$, the morphism complex \\ $\cC \otimes \cD((c,d),(c',d')) := \cC(c,c') \otimes_k \cD(d,d')$.
\end{itemize}

\end{defn}

For us, an important aspect of the model category structure is that it does not respect the monoidal structure of dg categories given by the pointwise tensor product. For instance, the following is known \cite[Exercise 14]{To2}.
\begin{ex}
Let $\Delta^1_k$ be the dg category with two objects 0 and 1 and with $\Delta^1_k(0,0)=k$, $\Delta^1_k(1,1)=k$, $\Delta^1_k(0,1)=k$, $\Delta^1_k(1,0)=0$.
Then $\Delta^1_k$ is a cofibrant dg category.  However, $\Delta^1_k \otimes \Delta^1_k$ is not a cofibrant dg category.  
\end{ex}

In the sequel, we will derive the pointwise tensor product of dg categories.

\subsection{The Morita Model Category Structure for Differential Graded Categories}\mbox{}\\
In this section we recall that there is a model category structure on $Cat_{dg}^k$ where the weak equivalences are the Morita equivalences.  Namely, we further localize the model category on $Cat_{dg}^k$ given in the previous section by the Morita equivalences.  We refer to this model category structure as the Morita model category structure on $Cat_{dg}^k$.  The Morita model category structure will be the one we use in our comparison theorem with $k$-linear stable $\oo$-categories. We begin by recalling several definitions in the theory of modules over a dg category.

\subsubsection{Differential Graded Categories of Modules}\mbox{}\\
We give a brief account of dg categories of modules which is largely based on \cite[2.3]{To2}.

\begin{defn} Given a dg category $\cT$ the category of right $\cT$-modules, $\hat{\cT}$ is defined to be $\Fun_{dg}(\cT^{op},Ch(k))$. That is, a right $\cT$-module is a functor from $\cT^{op} \to Ch(k)$.
\end{defn}

\begin{ex}
If $\cT$ is a dg category with one object, then $\cT$ has the same data as a dg algebra $A$.  The category right $\cT$-modules is equivalent to the category of right dg modules over the dg algebra $A$.
\end{ex}

\begin{rem} The category $\hat{\cT}$ can be given the structure of a combinatorial model category in which the weak equivalences are the pointwise equivalences in $Ch(k)$ and the fibrations are the pointwise fibrations in $Ch(k)$.  The generating cofibrations and generating acyclic cofibrations are the maps $\cT(-, x) \otimes f $ for $x \in \cT$ and $f$
varying through the generating cofibrations and generating acyclic cofibrations, respectively, of the model structure on chain complexes.  The representable right $\cT$-modules $\cT(-,x)$ are both cofibrant and compact.
\end{rem}
We will denote by $\hat{\cT}^{cf}$ the category of fibrant and cofibrant right $\cT$-modules, and by $\cD(\cT)$ the derived category of the dg category $\cT$.  As usual, there is an equivalence $[\hat{\cT}^{cf}] \simeq \cD(\cT)$. Thus, $\hat{\cT}^{cf}$ can be viewed as the dg enhancement of the derived category of $\cT$.

\begin{ex}
If $\cT$ is a dg category with one object, then $\cT$ has the same data as a dg algebra $A$.  The category $\cD(\cT)$ is the derived category of the dg algebra $A$.
\end{ex}

\subsubsection{Pretriangulated Differential Graded Categories}\mbox{}\\
We give a brief account of pretriangulated dg categories. Our definition for a pretriangulated dg category is a version of the definition of a pretriangulated spectral category in \cite[5.4]{BM}.

\begin{defn} \label{pretri1}
A dg category $\cT$ is called pretriangulated if 
\begin{enumerate}
\item There is an object $0$ in $\cT$ such that the right $\cT$-module $\cT(-,0)$ is homotopically trivial (weakly equivalent to the constant functor with value the
complex $0$).
\item Whenever a right $\cT$-module $M$ has the property that $\Sigma M$ is weakly equivalent to a representable $\cT$-module $\cT(-, c)$ (for some object $c \in \cT$), then $M$ is weakly equivalent to a representable $\cT$-module $\cT(-, d)$ for some object $d \in \cT$. (The suspension of a $\cT$-module $M$ refers to a pointwise shift when evaluated on each object of $\cT$).
\item Whenever the right $\cT$-modules $M$ and $N$ are weakly equivalent to representable $\cT$-modules $\cT(-, a)$ and $\cT(-, b)$ respectively, the homotopy
cofiber of any map of right $\cT$-modules $M \to N$ is weakly equivalent to a representable $\cT$-module.
\end{enumerate}
\end{defn}

\begin{ex}
If $X$ is a scheme over $k$, then its category of unbounded complexes of sheaves of $\cO_X$-modules is a pretriangulated dg category.
\end{ex}

\begin{rem}
Note that this definition promises a $0$ object in $Ho(\cT)$, as well as shift functors on $Ho(\cT)$.  
\end{rem}

We now relate our notion of pretriangulated, \ref{pretri1}, to that introduced by Bondal and Kapranov in \cite[Definition 2]{BK} and Drinfeld \cite[2.4]{Drinfeld}.
\begin{defn} \label{pretri2} To a dg category $\cA$, Bondal and Kapranov associate another dg category, $\cA^{pre-tr}$, where the objects are formal expressions 
$$\{ (\oplus_{i=1}^n C_i[r_i],q), C_i \in \cA, r_i \in \mathbb{Z}, q=(q_{ij}) \},$$
 where each $q_{ij} \in \Hom(C_i,C_j)[r_j-r_i]$ is homogenous of degree 1 and $dq+q^2=0$. If $C,C' \in \cA^{pre-tr}$, $C=(\oplus_{j=1}^n C_j[r_j],q)$ and $C'=(\oplus_{i=1}^n C'_i[r'_i],q')$, then the complex $\Hom(C,C')$ is the space of matrices $f=(f_{ij}), f_{ij} \in \Hom(C_j,C'_i)[r'_i-r_j]$ and composition is given by matrix multiplication. The differential $d: \Hom(C,C') \to \Hom(C,C')$ is given by 
 $$df := d'f+q'f- (-1)^lfq, \mbox{ if deg} (f_{ij})=l,$$ 
where $d' := (df_{ij})$. \\
Given a morphism $f: A \to B $ in $\cA$, the cone on f, $\text{cone}(f)$, is defined to be the object $(A[1] \oplus B, q) \in A^{pre-tr}$, where $q_{12}=f$ and $q_{11}=q_{21}=q_{22}=0$.\\
The dg category $\cA$ is said to be pretriangulated if $0 \in \cA^{pre-tr}$ is homotopic to an object in $\cA$, if for every $A \in \cA$, the object $A[n] \in A^{pre-tr}$, is homotopic to an object in $\cA$, and for every closed morphism $f$ of degree 0 in $\cA$, $\text{cone}(f) \in \cA^{pre-tr}$, is homotopic to an object in $\cA$.
\end{defn}

The concept behind definitions \ref{pretri1} and \ref{pretri2} of a pretriangulated dg category is for our dg categories to be pointed, closed under shifts, and closed under the formation of mapping cones. Thus
\begin{prop} The two notions of pretriangulated dg categories \ref{pretri1} and \ref{pretri2} agree.
\end{prop}
\begin{proof} 
There is a functor $$\alpha: \cA^{pre-tr} \to Fun_{dg}(\cA^{op},Ch(k))$$
 by sending $K = (\oplus_{i=1}^n C_i[r_i],q)$ to the functor $\alpha(K): \cA^{op} \to Ch(k)$, where $\alpha(K)(B) = \oplus \Hom_{\cA}(B, C_i)[r_i]$ with differential $d + q$, where $d$ is the differential on $\oplus \Hom_{\cA}(B, C_i)[r_i]$.  \\
The functor $\alpha$ is an embedding of $\cA^{pre-tr}$ as a full dg subcategory of $\hat{\cA}$ and it sends the element $\text{cone}(f) \in \cA^{pre-tr}$ to the element $\text{cone}(\alpha(f)) \in \hat{\cA}$. 
If $K$ is homotopic to an object in $K' \in \cA$, then $\alpha(K)$ is certainly weakly representable by $\cA( -, K')$. Thus, \ref{pretri2} implies \ref{pretri1}. 
Conversely, if the functor $\alpha(K)$ is weakly representable by an object $K' \in \cA$, then $K$ is homotopic to $K'$ in $\cA^{pre-tr}$. This follows since $\cA( -, K')$ is a cofibrant-fibrant object of $\hat{\cA}$ and $\cA^{pre-tr}$ is a full subcategory of $\hat{\cA}$. Thus, a weak homotopy equivalence $\cA( -, K') \to \alpha(K)$ is a homotopy equivalence in $\cA^{pre-tr}$. We conclude that \ref{pretri1} implies \ref{pretri2}. 
\end{proof}

\begin{cor} Since the homotopy category of a pretriangulated dg category as in \ref{pretri2} is triangulated by \cite[Proposition 2]{BK}, so is the homotopy category of a pretriangulated dg category as in \ref{pretri1}.
\end{cor}

We now wish to say that any small dg category embeds in a small pretriangulated category.  However, we must first define an embedding in the dg category setting.
\begin{defn}
A dg functor $\cF: \cT \to \cT'$ is a Dwyer-Kan embedding or DK-embedding if for any objects $x,y \in Ob(\cT)$, the map $\cT(x,y) \to \cT'(\cF (x),\cF (y))$ is a quasi-isomorphism.
\end{defn}

\begin{rem} This is equivalent to the definition of quasi-fully faithful in \cite[2.3]{To2}.
\end{rem}

\begin{prop} Any small dg category DK-embeds in a small pretriangulated dg category.
\end{prop}
\begin{proof} We use an argument similar to \cite[5.5]{BM}. Given any dg category $\cT$, its category of modules, $\hat{\cT}$, is pretriangulated. We restrict to a small full subcategory of $\hat{\cT}$ as follows: For any set $U$, write $U \hat{\cT}$ for the full subcategory of $\hat{\cT}$ consisting of functors taking values in $Ch(k)$ whose underlying sets are in $U$.  Then $U \hat{\cT}$ is a small dg category, and if we choose $U$ to be the power set of a sufficiently large cardinal, then $U \hat{\cT}$ will be closed under the usual constructions of homotopy theory in $\hat{\cT}$, including the small objects argument constructing factorizations. In particular, $U \hat{\cT}$ is a model category with cofibrations, fibrations, and weak equivalences the maps that are such in $\hat{\cT}$.  We also have a closed model category $\Fun_{dg}(U \hat{\cT}, Ch(k))$ of modules over $U \hat{\cT}$.  

Let $\tilde{\cT}$ be the full subcategory of $U \hat{\cT}$ consisting of the cofibrant-fibrant objects. This will be our desired small pretriangulated dg category.

We have a closed model category $\Fun_{dg}(\tilde{\cT}, Ch(k))$ of modules over $\tilde{\cT}$. Properties (1) and (2) for $\tilde{\cT}$ in the definition of a pretriangulated dg category are clear. For property (3), consider a map of $\tilde{\cT}$-modules $\cM \to \cN$. Since the model structure on $\Fun_{dg}(\tilde{\cT}, Ch(k))$ is left proper, after replacing $\cM$ and $\cN$ with fibrant replacements, we obtain an equivalent homotopy cofiber, and so we can assume without loss of generality that $\cM$ and $\cN$ are fibrant. We assume that $\cM$ is weakly equivalent to $\tilde{\cT}(-,a)$ and $\cN$ is weakly equivalent to $\tilde{\cT}(-,b)$ for objects $a,b \in \tilde{\cT}$. Since $\tilde{\cT}(-,a)$ and $\tilde{\cT}(-,b)$ are cofibrant and $\cM$ and $\cN$ are fibrant, we can choose weak equivalences $\tilde{\cT}(-,a) \to \cM$ and $\tilde{\cT}(-,b) \to \cN$. Furthermore, as $\tilde{\cT}(a,b)$ and $\cN(a)$ are both fibrant, we can lift the composite map $\tilde{\cT}(-,a) \to \cN$ to a homotopic map $\tilde{\cT}(-,a) \to \tilde{\cT}(-,b)$. We obtain a weak equivalence on the homotopy cofibers.  The map $\tilde{\cT}(-,a) \to \tilde{\cT}(-,b)$ is determined by the map $a \to b$ by the Yoneda lemma.  A fibrant replacement of the homotopy cofiber in $U \hat{\cT}$ is in $\tilde{\cT}$ and represents the homotopy cofiber of $\cM \to \cN$ in $\Fun_{dg}(\tilde{\cT}, Ch(k))$.
\end{proof}

Moreover, we have that
\begin{prop} A dg functor between pretriangulated dg categories induces a triangulated functor on its homotopy categories. Moreover, a dg functor between pretriangulated dg categories is a DK-equivalence if and only if it induces an equivalence of homotopy categories.
\end{prop}

Now that we have developed the notion of pretriangulated dg categories, we wish to use a model category structure on dg categories where every fibrant dg category is in particular pretriangulated.  The following definitions lie in the model for $\hat{\cT}$ constructed in the proof above.

\begin{defn} Given a dg category $\cT$, a right $\cT$-module $X: \cT^{op} \to Ch(k)$ is called a finite cell object if it can be obtained from the initial $\cT$-module by a finite sequence of pushouts along generating cofibrations in $\hat{\cT}$.
\end{defn}

\begin{defn} Given a dg category $\cT$, its triangulated closure, $\hat{\cT}_{tri}$, is defined to be the full dg subcategory of $\hat{\cT}^{cf}$ of objects that have the homotopy type of finite cell objects. Given a dg category $\cT$, the category of perfect modules, $\hat{\cT}_{perf}$, is the thick closure of $\hat{\cT}_{tri}$. That is, it is the smallest full dg subcategory of $\hat{\cT}^{cf}$ containing objects that have the homotopy type of retracts of finite cell objects.
\end{defn}

\begin{rem} The category $\hat{\cT}_{perf}$ is an idempotent-complete pretriangulated dg category. 
\end{rem}

\begin{defn}
A dg functor $\cF: \cT \to \cT'$ is called a Morita equivalence if  $\cF$ induces a DK-equivalence of dg categories $\hat{\cT}_{perf} \to \hat{\cT'}_{perf}$.  
\end{defn}  

Using the definitions above, the following is proved in \cite[5.1]{TabM}.
\begin{prop} The category $Cat^k_{dg}$ admits the structure of a combinatorial model category whose weak equivalences are the Morita equivalences and whose cofibrations are the same as those in the DK-model structure on $Cat^k_{dg}$.
\end{prop}

\begin{defn} The model category structure defined above will be referred to as the Morita model category structure on $Cat_{dg}^k$.
\end{defn}

\begin{rem} \label{fibrant replacement} The functor $R: Cat^k_{dg} \to Cat^k_{dg}$ given by $$\cT \mapsto \hat{\cT}_{perf}$$ is a fibrant replacement functor in the Morita model structure on $Cat^k_{dg}$.  
\end{rem}

The following property of the fibrant replacement functor $R$ is used crucially in \ref{newcomparison}. 
\begin{prop} \label{fibrant replacement preserves finite colimits} The value of a morphism in $Cat_{R \dash \Mod}$ under the functor $R$ preserves finite colimits and tensors of objects with finite spectra.
\end{prop}

\begin{proof} Let $\cF:\cC \to \cD$ be a morphism in $Cat_{R \dash \Mod}$.  Let  $R_{\cC} : \cC \to \hat{\cC}_{perf}$ and $R_{\cD} : \cD \to \hat{\cD}_{perf}$ 
be the embeddings of $\cC$ and $\cD$ into their categories of perfect modules. The functor $R_{\cF}: \hat{\cC}_{perf} \to \hat{\cD}_{perf}$ is given by the left Kan extension
\[\xymatrix{
\cC \ar[r]^{\cF} \ar[d]_-{R_{\cC}} & \cD \ar[r]^-{R_{\cD}} & \hat{\cD}_{perf} \\
\hat{\cC}_{perf} \ar@{-->}[urr]_-{R_{\cF}}
}\]
That is, $R_{\cF} := Lan_{R_{\cC}}(R_{\cD} \circ \cF)$. Thus, $R_{\cF}$ preserves finite colimits and tensors of objects with finite spectra by the universal property of the left Kan extension.
\end{proof}

\section{Spectral Categories and the Enriched Dold-Kan Correspondence}\mbox{}\\
The theory of spectral categories has a similar flavor to the theory of dg categories.  The reader is encouraged to read \cite{BGT1}, \cite{BGT2}, and \cite{BM} for further background in this direction.  

\subsection{Review of Spectral Categories}\mbox{}\\
We let $\cS$ denote the symmetric monoidal simplicial model category of symmetric spectra \cite{HSS}. When relevant, we will be using the stable model structure on $\cS$.
\begin{defn} A spectral category $\cA$ is given by
\begin{itemize}
\item A class of objects $Ob(\cA)$.
\item For each pair of objects $x,y \in Ob(\cA)$, a symmetric spectrum $\cA(x,y)$.
\item For each triple $x,y,z \in Ob(\cA)$ a composition in $\cS$,\\ $\cA(y,z) \wedge \cA(x,y) \to \cA(x,z)$.
\item For any $x \in Ob(\cA)$, a map $e_x: \mathbb{S} \to \cA(x,x)$ in $\cS$.
\end{itemize}
satisfying the usual associativity and unit conditions.  
\end{defn}

A spectral category is said to be small if its class of objects forms a set. We write $Cat_{\cS}$ for the category of small spectral categories and spectral (enriched) functors. As with dg categories, there is a DK-model structure on $Cat_{\cS}$, where the weak equivalences are the DK-equivalences \cite[5.10]{TabS}.

\begin{defn}
Let $\cA$ be a spectral category. The homotopy category of $\cA$, denoted by $[\cA]$, is the ordinary category given by:
\begin{itemize}
\item The objects of $[\cA]$, $Ob([\cA])$, are the objects of $\cA$.
\item For every pair of objects $x,y \in Ob([\cA])$, the set of morphisms $[\cA](x,y) := \pi_0(\cT(x,y))$.
\end{itemize}
\end{defn}

\begin{defn}
A spectral functor $\cF: \cA \to \cA'$ is a DK-equivalence if 
\begin{itemize} 
\item For every pair of objects $x,y \in Ob(\cA)$, the morphism in $\cS$, $\cF(x,y): \cA(x,y) \to \cB(x,y)$ is a stable equivalence of symmetric spectra.
\item The induced functor $[\cF]: [\cA] \to [\cB]$ is an equivalence of ordinary categories.
\end{itemize}
\end{defn}

Moreover, we have \cite[2.2.4]{BGT2}.
\begin{prop} The category $Cat_{\cS}$ with the DK-model structure is a combinatorial model category. Moreover, $Cat_{\cS}$ can be replaced by a Quillen equivalent simplicial model category.
\end{prop}

\begin{rem} As with dg categories, there are similar notions of a pretriangulated spectral category and a Morita equivalence of spectral categories.  We can use the machinery of Bousfield localization in this setting to obtain a combinatorial model category structure on $Cat_{\cS}$ where the weak equivalences are the Morita equivalences. 
\end{rem}

Since this model structure is the one we will use in our theorem, we record here as a proposition.

\begin{prop} The category $Cat_{\cS}$ with the Morita model structure is a combinatorial model category. 
\end{prop} 

\subsection{Review of $R$-Module Spectra}\mbox{}\\
In this section we set up the relevant properties of the model category of $R$-module spectra used in our main theorem. Let $R$ be an $E_{\oo}$-ring spectrum and $R \dash \Mod$ the category of $R$-module spectra. 

\begin{prop}
There is a cofibrantly generated symmetric monoidal model category structure on $R \dash \Mod$  by \cite{HSS}.
\end{prop}

\begin{rem}
More explicitly, start with the stable model category on symmetric spectra \cite[6.3]{HSS}. Applying the machinery of \cite{SS1}, one then constructs a model category structure on modules over $R$. Since the former is cofibrantly generated, so is the latter. The generating cofibrations and generating acyclic cofibrations for symmetric spectra are given by the sets $FI_{\partial}$ and $K \cup FI_{\Lambda}$ respectively.  Here, 
$FI_{\partial}$ is the set of maps
$$F_n\partial \Delta [m]^{+} \to F_n \Delta [m]^{+} ,$$
$K$ is the set of pushouts products
$$\Delta [m]^{+} \wedge F_{n+1} S^1 \cup_{\partial \Delta [m]^{+} \wedge F_{n+1}S^1} \partial  \Delta [m]^{+} \wedge Z(\lambda_n) \to  \Delta [m]^{+} \wedge Z(\lambda_n),$$
and $FI_{\Lambda}$ is the set of horn inclusions 
$$F_n \Lambda^k [m]^{+} \to F_n \Delta [m]^{+}.$$  
The functor $F_n$ is the left adjoint of the evaluation functor $Ev_n$ and $Z(\lambda_n)$ is the mapping cylinder of the natural map $F_{n+1}S^1 \to F_nS^0$. The generating cofibrations for $R \dash \Mod$ are given by the set $R \wedge FI_{\partial}$ and the generating acyclic cofibrations are given by $R \wedge (K \cup FI_{\Lambda})$.  For a more detailed description, please see \cite[3.4]{HSS}.
\end{rem}

Much of the structure of the model category $R \dash \Mod$ is controlled by the unit object $R \in R \dash \Mod$, and it is convenient to describe how $R$-modules can be built from the unit object $R$ via operations such as taking various colimits and retracts thereof.
\begin{defn} An $R$-module $M$ is perfect if it lies in the smallest stable subcategory of $R \dash \Mod$ containing $R$ and closed under finite homotopy colimits and retracts.
\end{defn}

We will use the notion of cell $R$-modules as a way to gain control over the cofibrant objects in $R \dash \Mod$.
\begin{defn} \label{cell1} An $R$-module $M$ is cell if $M$ is the union of an expanding sequence of sub $R$-modules $M_l$ such that $M_0 = *$ and $M_{j+1}$ is the homotopy cofiber of a map $\phi_j: F_j \to M_j$, where $F_j$ is a wedge of shifts of $R$ as a module over itself.  If a cell $R$-module is additionally perfect, we will refer to it as a perfect cell $R$-module.
\end{defn}

\begin{nota} We denote the category of cell $R$-modules by $R \dash \Mod^{\cell}$ and the category of perfect cell $R$-modules by $\Perf(R)^{\cell}$.  Since the smash product over $R$ of two cell $R$-modules is a cell $R$-module, $R \dash \Mod^{\cell}$ and $\Perf(R)^{\cell}$ are symmetric monoidal spectral categories with unit $R$.
\end{nota}

\begin{rem} The definition of cell $R$-modules given above agrees with the usual definition of cell objects using the model structure in the stable model category of symmetric spectra mentioned in the previous remark.  
\end{rem}

We will also need a crucial property of cell $R$-modules.
\begin{prop}
Every $R$-module is functorially weakly equivalent to a cell $R$-module. Moreover, every perfect $R$-module is functorially weakly equivalent to a perfect cell $R$-module.
\end{prop}

\begin{proof} Retracts of cell objects are the cofibrant objects in any cofibrantly-generated model category. Moreover, the standard use of the small object argument can be used to make this choice functorial.
\end{proof}

We conclude this section with the following.
\begin{prop} The category $\Perf(R)^{\cell}$ is generated by $R$ under finite colimits and tensors with finite spectra.
\end{prop}
\begin{proof} Suspensions and desuspensions of $R$ are obtained by tensoring $R$ with a shifted sphere spectrum $\SS[n] := \Sigma^{n} \SS$. That is,
$R[n] \simeq R \wedge \SS[n]$. By the definition above, a perfect cell $R$-module $M$ can be expressed as a finite colimit over a diagram consisting of objects of the form $R[n]$. The result follows.
\end{proof}

\subsection{Comparing Categories Enriched in $R \dash \Mod$ and $\Perf(R)^{\cell}$-Module Categories} \label{enrichedaremodules}
In this section, we will compare small categories enriched in $R \dash \Mod$ and module categories over the symmetric monoidal category $\Perf(R)^{\cell}$.
Using the symmetric monoidal model category structure on $R \dash \Mod$ we can define the model category of categories enriched in $R \dash \Mod$ mimicking the definition of the model structure on categories enriched in spectra.
 
\begin{defn} Let $Cat_{R \dash \Mod}$ denote the combinatorial model category of categories whose morphism spaces are enriched in $R \dash \Mod$.  
\end{defn}

There is a monoidal structure on $Cat_{R \dash \Mod}$ given by the pointwise smash product.  

\begin{defn} Given $\cC,\cD \in Cat_{R \dash \Mod}$. The pointwise smash product $\cC \wedge \cD$ is specified by
\begin{itemize}
\item The objects of $\cC \wedge \cD$, $Ob( \cC \wedge \cD )$, are given by pairs of objects $(c,d)$, where $c \in \cC$ and $d \in \cD$.
\item  For every pair of objects $(c,d) , (c',d') \in \cC \wedge \cD$, the morphism spectrum $\cC \wedge \cD((c,d),(c',d') := \cC(c,c') \wedge_R \cD(d,d')$.
\end{itemize}
\end{defn}

\begin{defn} A module category over the symmetric monoidal category $\Perf(R)^{\cell}$ is a spectral category $\cM$ with a spectral bifunctor $$\alpha: \Perf(R)^{\cell} \wedge \cM \to \cM$$
preserving finite colimits and tensors of objects with finite spectra in each variable, and with 
\begin{itemize}
\item Functorial associativity isomorphisms $m_{x,y,m}: \alpha(x \wedge_R y, m) \to \alpha(x,  \alpha(y,m))$.
\item Unit isomorphisms $l_{m}: \alpha(R, m) \to m$.
\end{itemize}
for any $x,y \in \Perf(R)^{\cell}$ and $m \in \cM$.
The isomorphisms above are subject to the usual pentagon and triangle axioms. 
\end{defn}

\begin{defn} A functor $\cF:\cM \to \cM'$ between two module categories over $\Perf(R)^{\cell}$ is a spectral functor equipped with 
functorial isomorphisms $$c_{x,M}: \cF(\alpha(x, m)) \to \alpha(x, \cF(m))$$
for any $x \in \Perf(R)^{\cell}$ and $m \in \cM$. These isomorphisms are subject to the usual pentagon and triangle axioms.
\end{defn}

We can now define
\begin{defn}
Let $\Mod_{\Perf(R)^{\cell}}(Cat_{\cS})$ be the category of module categories over the category $\Perf(R)^{\cell}$. A morphism in 
$\Mod_{\Perf(R)^{\cell}}(Cat_{\cS})$ is defined to be a Morita equivalence if it is a Morita equivalence of underlying spectral categories. 
\end{defn}

\begin{nota}  We now consider subcategories of $Cat_{R \dash \Mod}$ and $\Mod_{\Perf(R)^{\cell}}(Cat_{\cS})$ consisting of categories closed under all finite colimits and tensors of objects with finite spectra.
\begin{itemize} 

\item Let $Cat^{\fincol}_{R \dash \Mod}$ be the category with objects categories enriched in $R$-modules closed under all finite colimits and tensors with finite spectral. Morphisms are functors of categories enriched in $R$-modules which preserve finite colimits and tensors of objects with finite spectra.

\item Let $\Mod_{\Perf(R)^{\cell}}(Cat^{\fincol}_{\cS})$ be the category with objects module categories over $\Perf(R)^{\cell}$ closed under all finite colimits and tensors with finite spectra. Morphisms are functors of module categories over $\Perf(R)^{\cell}$ which preserve finite colimits and tensors of objects with finite spectra.

\end{itemize}
\end{nota}
 
We have the following:
\begin{prop} \label{comparison}
There is an equivalence of categories
$$Cat^{\fincol}_{R \dash \Mod} \simeq \Mod_{\Perf(R)^{\cell}}(Cat^{\fincol}_{\cS})$$
 Moreover, this equivalence is the identity on underlying spectral categories.
\end{prop}
\begin{proof}
Let $\cC$ be a category in $\Mod_{\Perf(R)^{\cell}}(Cat^{\fincol}_{\cS})$, then $\cC$ is equipped with a bifunctor $\alpha: \Perf(R)^{\cell} \wedge \cC \to \cC$, that preserves finite colimits in each variable and has isomorphisms $\alpha(R,x) \simeq x$ for all objects $x \in \cC$.  We see that $\cC$ is thus equipped with maps $Hom_R(R,R) \wedge \cC(x,y) \to \cC(x,y)$, which is equivalent to $R \wedge \cC(x,y) \to \cC(x,y)$, so each mapping space in $\cC$ is indeed an $R$-module. Thus, $\cC \in Cat^{\fincol}_{R \dash \Mod}$.  Moreover, given a morphism $\cG: \cC \to \cD$ in $\Mod_{\Perf(R)^{\cell}}(Cat^{\fincol}_{\cS})$ restriction to $*/R$ gives $\cG$ the structure of a morphism in $Cat^{\fincol}_{R \dash \Mod}$. 

Thus, this construction defines a functor 
$$\Res: \Mod_{\Perf(R)^{\cell}}(Cat^{\fincol}_{\cS}) \to Cat^{\fincol}_{R \dash \Mod}.$$

Now, let $\cC' \in Cat^{\fincol}_{R \dash \Mod}$ be closed under finite colimits and tensors of objects with finite spectra. We wish to produce a functor $\alpha: \Perf(R)^{\cell} \wedge \cC' \to \cC'$. The data of the action of $R$ on each mapping spectrum $\cC'(x,y)$ is equivalent to the data of an action map $act: */R \wedge \cC' \to \cC'$, where $*/R$ is the spectral category with one object $*$ and $\Hom_{*/R}(*,*)=R$.  That is, $\cC'$ is a module category over $*/R$.

Starting with the action $act: */R \wedge \cC' \to \cC'$ and the full and faithful inclusion $i: */R \wedge \cC' \to \Perf(R)^{\cell} \wedge \cC'$, let 
$Lan_i(act): \Perf(R)^{\cell} \wedge \cC' \to \cC'$ be the spectral enriched left Kan extension of $act$ by $i$ \cite{Du}. This enriched left Kan extension exists because $\cC'$ closed under finite colimits and tensors of objects with finite spectra. Thus, we have the following commutative diagram satisfying the universal property of left Kan extensions:
\[\xymatrix{
\ast/R \wedge \cC' \ar[r]^-{act} \ar[d]_-{i} & \cC' \\
\Perf(R)^{\cell} \wedge \cC' \ar[ur]_-{Lan_i(act)}.
}\]

We claim the functor $Lan_i(act)$ is the desired functor $\alpha: \Perf(R)^{\cell} \wedge \cC' \to \cC'$.  The condition that the bifunctor $\alpha$ preserves colimits and tensors of objects by finite spectra in each variable follows from the definition of an enriched left Kan extension. That $\alpha$ is associative, i.e that $$\alpha(M \wedge_R N,x) \simeq \alpha(M,\alpha(N,x))$$ for all $M,N \in \Perf(R)^{\cell}$ and $x \in \cC'$ also follows from the definitions. Namely, write $M$ and $N$ as finite colimits $M \simeq \underset{i}{\mbox{colim}}R[n_i]$ and 
$N \simeq \underset{j}{\mbox{colim}}R[n_j]$.  Then we have 
\begin{align*}
\alpha(M \wedge_R N,x) &\simeq \alpha((\underset{i}{\mbox{colim}}R[n_i]) \wedge_R (\underset{j}{\mbox{colim}}R[n_j]),x) \simeq \alpha(\underset{i,j}{\mbox{colim}}(R[n_i] \wedge_R R[n_j]),x) \\
&\simeq  \underset{i,j}{\mbox{colim}}( \alpha(R[n_i+n_j],x)) \simeq  \underset{i,j}{\mbox{colim}}( act(R,x)[n_i+n_j])  \\ 
& \simeq \underset{i}{\mbox{colim}} \thinspace \underset{j}{\mbox{colim}} \thinspace \alpha(R[n_i], \alpha(R[n_j],x) \\
&\simeq \underset{i}{\mbox{colim}} \thinspace \alpha(R[n_i],  \underset{j}{\mbox{colim}} \thinspace \alpha(R[n_j],x)) \\
&\simeq  \alpha(\underset{i}{\mbox{colim}} \thinspace R[n_i],  \alpha(\underset{j}{\mbox{colim}}\thinspace R[n_j],x)) \simeq \alpha(M, \alpha(N,x)).
\end{align*}

Now let $\cC',\cD' \in Cat^{\fincol}_{R \dash \Mod}$ have all finite colimits and be closed under tensors with finite spectra. Let $\cF:\cC' \to \cD'$ be a functor in $Cat^{\fincol}_{R \dash \Mod}$ which preserves finite colimits and tensors of objects by finite spectra. We can use the left Kan extension above to construct $\alpha : \Perf(R)^{\cell} \wedge \cC' \to \cC'$ and $\beta: \Perf(R)^{\cell} \wedge \cD' \to \cD'$: structures of a module category over $\Perf(R)^{\cell}$. We wish to show that $\cF$ also extends to a functor in $\Mod_{\Perf(R)^{\cell}}(Cat_{\cS})$. Re-expressing the functor $\cF$ as a diagram
\[\xymatrix{
\ast/R \wedge \cC'  \ar[d]^{\Id \wedge \cF} \ar[r] &\cC' \ar[d]^{\cF} \\
\ast/R \wedge \cD' \ar[r] & \cD'.
}\]
We wish show this extends to the following diagram: 
\[\xymatrix{
\Perf(R)^{\cell} \wedge \cC'  \ar[d]^{\Id \wedge \cF} \ar[r]^-{\alpha} &\cC' \ar[d]^{\cF} \\
\Perf(R)^{\cell} \wedge \cD' \ar[r]^-{\beta} & \cD'.
}\]
This also follows from the definition of the left Kan extension. Namely, let $x \in \cC'$ and let $M \in \Perf(R)^{\cell}$ so that $M \simeq \underset{i}{\mbox{colim}}R[n_i]$, where this colimit is finite.
Then we have
\begin{align*}
\cF(\alpha(M,x)) &\simeq \cF(\alpha(\underset{i}{\mbox{colim}}R[n_i],x)) \\
&\simeq \cF(\underset{i}{\mbox{colim}} \alpha(R[n_i],x)) \\
&\simeq \underset{i}{\mbox{colim}} \cF(\alpha(R[n_i],x)) \\
&\simeq \underset{i}{\mbox{colim}} \beta(R[n_i],\cF(x)) \\
&\simeq \beta(\underset{i}{\mbox{colim}}R[n_i],\cF(x)) \\
&\simeq \beta(M,\cF(x)) 
\end{align*}
as desired.

Thus, this construction of left Kan extension defines a functor 
$$\LKan: Cat^{\fincol}_{R \dash \Mod} \to \Mod_{\Perf(R)^{\cell}}(Cat^{\fincol}_{\cS}).$$

We now wish to show that $\Res$ and $\LKan$ are inverse equivalences. It is easy to see there is an equivalence $\Id \simeq \Res \circ \LKan$.  Thus, it suffices to 
show there is an equivalence $$\LKan \circ \Res \simeq \Id.$$

If we start with $\cC'$ a module category of $\Perf(R)^{\cell}$ closed under finite colimits and tensors of objects by finite spectra
$$\beta: \Perf(R)^{\cell} \wedge \cC' \to \cC',$$ restrict to the action on each hom spectra $act_{\beta}: R \wedge \cC'(x,y) \to \cC'(x,y)$, then we will show the induced natural transformation $Lan_i(act_{\beta}) \to \beta$ is a natural isomorphism.  

Let $M \in \Perf(R)^{\cell}$ so that $M \simeq \underset{i}{\mbox{colim}}R[n_i]$, where this colimit is finite. Then, for $x \in \cC'$ the natural transformation
\[\xymatrix{ \ar @{} [dr] |{\Updownarrow}
Lan_i(act_{\beta})(M,x) \ar[r] & \beta(M,x) \\
Lan_i(act_{\beta})( \underset{i}{\mbox{colim}}R[n_i],x) \ar[r]& \beta( \underset{i}{\mbox{colim}}R[n_i],x)
}\]

is equivalent to
$$\underset{i}{\mbox{colim}} \thinspace act_{\beta}(R,x)[n_i] \to\underset{i}{\mbox{colim}} \beta(R,x)[n_i].$$
However, this is an isomorphism by definition.  

Thus, we have our desired equivalence $$Cat^{\fincol}_{R \dash \Mod} \simeq \Mod_{\Perf(R)^{\cell}}(Cat^{\fincol}_{\cS}).$$
\end{proof}

\begin{rem} 
Let us also describe how the enriched left Kan extension acts on morphism spectra.
Let $\cC' \in Cat^{\fincol}_{R \dash \Mod}$ and let $x_1,x_2 \in \cC'$. We start with an action map
$$act: R \wedge \Hom_{\cC'}(x_1,x_2) \to \Hom_{\cC'}(x_1,x_2).$$

Now let $M \simeq \underset{i}{\mbox{colim}}R[n_i]$ be a perfect cell $R$-module (i.e the colimit is finite). The enriched left Kan extension determines a map 
\[\xymatrix{ \ar @{} [dr] |{\Updownarrow} 
\Hom_{\Perf(R)^{\cell} \wedge \cC'} ((R,x_1),(M,x_2)) \ar[r]& \Hom_{\cC'}(\alpha(R,x_1),\alpha(M,x_2)) \\
\ar @{} [dr] |{\Updownarrow} \Hom_{\Perf(R)^{\cell}}(R,\underset{i}{\mbox{colim}}R[n_i]) \wedge \Hom_{\cC'}(x_1,x_2) \ar[r]& \Hom_{\cC'}(\alpha(R,x_1),\underset{i}{\mbox{colim}}\alpha(R[n_i],x_2)) \\
\underset{i}{\mbox{colim}}R[n_i] \wedge \Hom_{\cC'}(x_1,x_2) \ar[r]& \Hom_{\cC'}(x_1,\underset{i}{\mbox{colim}} \thinspace x_2[n_i]).
}\]
\end{rem}

We will return to the comparison between $Cat_{R \dash \Mod}$ and $\Mod_{\Perf(R)^{\cell}}(Cat_{\cS})$ after reviewing how to pass from model categories to $\oo$-categories. We will obtain an equivalence between the underlying $\oo$-categories of $Cat_{R \dash \Mod}$ and $\Mod_{\Perf(R)^{\cell}}(Cat_{\cS})$ after localizing by Morita equivalences.

\subsection{The Enriched Dold-Kan Correspondence}\mbox{}\\
In \cite{SS2}, Schwede-Shipley generalize the Dold-Kan correspondence and prove that the category of modules over the Eilenberg-MacLane symmetric spectrum $Hk$, $Hk \dash \Mod$, is Quillen equivalent to chain complexes of $k$-modules, $Ch(k)$. In \cite{Tab1}, Tabuada further generalizes this result and establishes a Quillen equivalence between dg categories over $k$, $Cat_{dg}^k$, and categories enriched in $Hk$-module symmetric spectra, $Cat_{Hk \dash \Mod}$, where both categories have the DK-model structure:

\begin{prop} There is a Quillen equivance
$$Cat_{dg}^k \simeq Cat_{Hk \dash \Mod}$$
where both categories are endowed with the DK-model structure.
\end{prop}

We will exploit this equivalence between dg categories and categories enriched in $Hk$-module spectra in our comparison theorem by using machinery that was developed for spectral categories in \cite{BGT1} and \cite{BGT2}.  

Moreover, one can conclude that:
\begin{prop} \label{pretridgsc}
A dg category $\cT$ is pretriangulated if and only if its associated spectral category is pretriangulated. 
Similarly, two dg categories are Morita equivalent if and only if their associated spectral categories are Morita equivalent.
\end{prop}

\begin{cor} There is a Quillen equivance
$$Cat_{dg}^k \simeq Cat_{Hk \dash \Mod}$$
where both categories are endowed with the Morita model structure.
\end{cor}

\section{Review of $\oo$-Categories} \label{reviewofoocats}\mbox{}\\
We now work with the theory of quasicategories, a well-developed model of $\oo$-categories. These first appeared in the work of Boardman and Vogt, where they were referred to as weak Kan complexes. The theory was subsequently developed by Joyal and then extensively studied by Lurie. In this section we give a brief review of the relevant background regarding the theory of $\oo$-categories. Our basic references for this material are Jacob Lurie's books \cite{HA},\cite{T}.
An extremely brief introduction to the definition of an $\oo$-category is given in \cite{L3}.

\subsection{From Model Categories to $\oo$-Categories} \label{modeloo}\mbox{}\\
There are a number of options for producing the ``underlying" $\oo$-category of a category equipped with a notion of ``weak equivalence." The most structured setting is that of a simplicial model category $C$, where the $\oo$-category can be obtained by restricting to the full simplicial subcategory $C^{cf}$ of cofibrant-fibrant objects and then applying the simplicial nerve functor $N$. More generally, if $C$ is a category equipped with a subcategory of weak equivalences $wC$ , the Dwyer-Kan simplicial localization $LC$ provides a corresponding simplicial category, and then $N((LC)^f)$, where $(-)^f$ denotes fibrant replacement in simplicial categories, yields an associated $\oo$-category. Lurie has given a version of this approach in \cite[1.3.4]{HA}: we associate to a (not necessarily simplicial) category $C$ with weak equivalences $W$ an $\oo$-category $N(C)[W^{-1}]$. Here, the notation $N(C)[W^{-1}]$ refers to the universal $\oo$-category equipped with a map $N(C) \to N(C)[W^{-1}]$ such that for another $\oo$-category $D$, the functor induced by precomposition 
$$\Fun(N(C)[W^{-1}],D) \to \Fun(N(C),D)$$ 
is a fully faithful embedding whose essential image is the collection of functors from $C$ to $D$ that map the image of morphisms in $W$ to equivalences in $D$ \cite[1.3.4.1]{HA}. If $C$ is a model category, it is usually convenient to restrict to the cofibrant objects $C^c$ and consider $N(C^c)[W^{-1}]$.

All of these constructions produce equivalent $\oo$-categories if $C$ is a model category \cite[1.3.4]{HA}. We will refer to this construction as the underlying $\oo$-category of a model category.

Furthermore, all of these constructions are functorial. In the sequel, we will need that given a Quillen adjunction $(F,G)$, there is an induced adjunction of functors on the level of the associated $\oo$-categories \cite[5.2.4.6]{T}. Moreover, a Quillen equivalence will induce an equivalence of $\oo$-categories. 

We will often be interested in $\oo$-categories for which we have set theoretic control.  Lurie provides a thorough treatment of the theory of presentable and accessible $\oo$-categories in \cite[5.4]{T} and \cite[5.5]{T}.  Briefly, an $\oo$-category $\cA$ is accessible if it is locally small and has a good supply of filtered colimits and compact objects. A great source of examples of accessible $\oo$-categories are Ind-categories (defined in the sequel) of small $\oo$-categories. An $\oo$-category $\cA$ is presentable if it is accessible and furthermore admits all small colimits.

\begin{nota} Let $\cP r^{L}$ denote the $\oo$-category of presentable $\oo$-categories and colimit-preserving functors; the $\oo$-category of colimit preserving functors is denoted $\Fun^{L}(-,-)$. In fact, $\Fun^{L}(-,-)$ is a presentable $\oo$-category and provides an internal hom object for $\cP r^{L}$.
\end{nota}

We will also need \cite[1.3.4.22]{HA}:
\begin{prop} Let C be a combinatorial model category, then the underlying $\oo$-category of C is a presentable $\oo$-category.
\end{prop}

We now use this proposition to show that the $\oo$-category $N(\Mod_{\Perf(R)^{\cell}}(Cat_{\cS}))[W'^{-1}]$  
is a presentable $\oo$-category, where $\Mod_{\Perf(R)^{\cell}}(Cat_{\cS})$ is introduced in Section \ref{enrichedaremodules}.

First, using the comparison \ref{comparison} we have the following proposition.
\begin{prop} \label{newcomparison} Let $W$ be the class of Morita equivalences in $Cat_{R \dash \Mod}$ and let $W'$ be the class of Morita equivalences in $\Mod_{\Perf(R)^{\cell}}(Cat_{\cS})$.
There is an equivalence of underlying $\oo$-categories
$$N(Cat_{R \dash \Mod})[W^{-1}] \simeq N(\Mod_{\Perf(R)^{\cell}}(Cat_{\cS}))[W'^{-1}].$$
\end{prop}
\begin{proof}
In \ref{comparison}, we proved an equivalence of categories
$$Cat^{\fincol}_{R \dash \Mod} \simeq \Mod_{\Perf(R)^{\cell}}(Cat^{\fincol}_{\cS})$$
This equivalence induces an equivalence of $\oo$-categories
$$N(Cat^{\fincol}_{R \dash \Mod})[W^{-1}] \simeq N(\Mod_{\Perf(R)^{\cell}}(Cat^{\fincol}_{\cS}))[W'^{-1}].$$
The natural inclusion functor 
$$i:Cat^{\fincol}_{R \dash \Mod} \to Cat_{R \dash \Mod}$$
induces an equivalence of $\oo$-categories
$$N(Cat^{\fincol}_{R \dash \Mod})[W^{-1}] \simeq N(Cat_{R \dash \Mod})[W^{-1}]$$
by \cite[1.3.4.16]{HA} since $Cat_{R \dash \Mod}$ (with its Morita model structure) has a fibrant replacement functor $R:Cat_{R \dash \Mod} \to Cat_{R \dash \Mod}$ which takes values in $Cat^{\fincol}_{R \dash \Mod}$ by \ref{fibrant replacement} and \ref{fibrant replacement preserves finite colimits}. 

Similarly, the natural inclusion functor 
$$i:Cat^{\fincol}_{\cS} \to Cat_{\cS}$$
induces an equivalence of $\oo$-categories
$$N(\Mod_{\Perf(R)^{\cell}}(Cat^{\fincol}_{\cS}))[W'^{-1}]  \simeq N(\Mod_{\Perf(R)^{\cell}}(Cat_{\cS}))[W'^{-1}].$$

In summary, we have the following diagram of equivalences of $\oo$-categories:
\[\xymatrix{
N(Cat^{\fincol}_{R \dash \Mod})[W^{-1}] \ar[d]_-{\simeq} \ar[r]^-{\simeq} & N(\Mod_{\Perf(R)^{\cell}}(Cat^{\fincol}_{\cS}))[W'^{-1}] \ar[d]^-{\simeq} \\
N(Cat_{R \dash \Mod})[W^{-1}] & N(\Mod_{\Perf(R)^{\cell}}(Cat_{\cS}))[W'^{-1}].
}\]
The conclusion follows.
\end{proof}

\begin{cor} \label{presentable}The $\oo$-category $N(\Mod_{\Perf(R)^{\cell}}(Cat_{\cS}))[W'^{-1}]$ is presentable.
\end{cor}
\begin{proof} The category $Cat_{R \dash \Mod}$ with the Morita model structure is a combinatorial model category, hence $N(Cat_{R \dash \Mod})[W^{-1}]$ is presentable.
Now use the proposition above.
\end{proof}

\subsection{Stable $\oo$-Categories and Idempotent-Complete $\oo$-Categories}
We now recall the definition of a stable $\oo$-category.  There is a close a connection between stable $\oo$-categories and spectral categories.  On the one hand, for every pair of objects in a stable $\oo$-category one can extract a mapping spectrum. On the other hand, given a category enriched in spectra, its category of right modules has a projective model structure and its associated $\oo$-category is stable.  

\begin{defn}
An $\oo$-category is stable \cite[1.1.1.9]{HA} if it has finite limits and colimits and pushout and pullback squares coincide \cite[1.1.3.4]{HA}. Let $Cat_{\oo}^{ex}$ denote the $\oo$-category of small stable $\oo$-categories and exact functors (i.e. functors which preserve finite limits and colimits) \cite[1.1.4]{HA}. The $\oo$-category of exact functors between $A$ and $B$ is denoted by $\Fun^{ex}(A, B)$.
\end{defn}

\begin{rem}
For a small stable $\oo$-category $\cC$, the homotopy category $Ho(\cC)$ is triangulated, with the exact triangles determined by the cofiber sequences in $\cC$ \cite[1.1.2.15]{HA}. This is why we use the Morita model structure on dg categories (and on $Cat_{Hk-Mod}$), i.e. fibrant objects in the Morita model structure on dg categories are pretriangulated.  
\end{rem}

Recall that an $\oo$-category $\cC$ is idempotent-complete if the image of $\cC$ under the Yoneda embedding $\cC \to \Fun(\cC, N(Set_{\Delta}^{cf}))$ is closed under retracts, where $N(Set_{\Delta}^{cf})$ is the $\oo$-category of spaces. 

\begin{nota} Let $Cat^{perf}_{\oo}$ denote the $\oo$-category of small idempotent-complete stable $\oo$-categories and exact functors. 
\end{nota}

There is an idempotent completion functor given as the left adjoint to the inclusion $Cat^{perf}_{\oo} \to Cat^{ex}_{\oo}$, which we denote by $\Idem(-)$.

\begin{defn} Let $\cA$ and $\cB$ be small stable $\oo$-categories. Then we will say that $\cA$ and $\cB$ are Morita equivalent if $\Idem(\cA)$ and $\Idem(\cB)$ are equivalent.
\end{defn}

It is shown in \cite[4.23]{BGT1} that:
\begin{thm} The $\oo$-category $Cat^{perf}_{\oo}$ is the underling $\oo$-category of the category $Cat_{\cS}$ endowed with the Morita model category structure. That is, the notion of Morita equivalence for spectral categories is compatible with the notion of Morita equivalence for stable $\oo$-categories.
\end{thm}

Thus, pretriangulated spectral categories can be interpreted as a rectified model of idempotent-complete small stable $\oo$-categories.

\subsection{$\Ind$-Categories}
Our review of $\Ind$-categories is based on the material in \cite[4.1]{BFN} and \cite[2.4]{BGT1}.  We will use $\Ind$-categories to define the symmetric monoidal structure on $Cat^{perf}_{\oo}$ and to relate large and small $\oo$-categories.

Presentable $\oo$-categories are large $\oo$-categories that are generated under sufficiently large filtered colimits by some small $\oo$-category. To make this notion precise we need the notion of an $\Ind$-category. Given a small $\oo$-category $\cC$, we can form the $\oo$-category $\cP(\cC)$ of presheaves on $\cC$ valued in the $\oo$-category of spaces. This is the formal closure of $\cC$ under colimits. There is a fully faithful Yoneda embedding $\cC \to \cP(\cC)$ and $\cP(\cC)$ is generated by the image of $\cC$ under small colimits \cite[5.1.5.8]{T}. For any $\oo$-category $\cC$ and infinite regular cardinal $\kappa$, we can form

\begin{defn}
The $\Ind$-category
$\Ind_\kappa(\cC)$, which is the formal closure under $\kappa$-filtered colimits of $\cC$ \cite[5.3.5]{T}. 
\end{defn}

The $\oo$-category $\Ind_{\kappa}(\cC)$ is a full subcategory of 
$\cP(\cC)$, and the Yoneda embedding $\cC \to \cP(\cC)$ factors as
$\cC \to \Ind_{\kappa}(\cC) \to \cP(\cC)$.  

We record here the following useful properties of the construction of the $\Ind$-category.

\begin{prop}
Let $\cC$ be a small $\oo$-category and $\kappa$ an infinite regular cardinal.
\begin{itemize}
\item The $\oo$-category $\Ind_{\kappa}(\cC)$ admits all $\kappa$-small colimits that exist in $\cC$ \cite[5.3.5.14, 5.5.1.1]{T}. 
\item The functor $\cC \to \Ind_{\kappa}(\cC)$ preserves $\kappa$-filtered colimits \cite[5.3.5.2,5.3.5.3]{T}. 
\item  If $\cC$ is additionally a stable $\oo$-category, then $\Ind_{\kappa}(\cC)$ is a stable $\oo$-category \cite[1.1.3.6]{HA}.
\item The image of $\cC$ in $\Ind_{\kappa}(\cC)$ provides a set of compact objects which generates $\Ind(\cC)$ under $\kappa$-filtered colimits \cite[5.3.5.5,5.3.5.11]{T}.
\item The category $\Ind_{\kappa}(\cC)$ is characterized by the property that it has $\kappa$-small filtered colimits, admits a functor $\cC \to \Ind_{\kappa}(\cC)$, and this functor induces an equivalence
\[
\Fun_{\kappa}(\Ind(\cC), \cD) \to \Fun(\cC,\cD), 
\]
for any $\cD$ which admits $\kappa$-filtered colimits (here
$\Fun_{\kappa}(-,-)$ denotes the $\oo$-category of functors that preserve
$\kappa$-small filtered colimits) \cite[5.3.5.10]{T}.
\end{itemize}
\end{prop}

The preceding discussion carries over when we restrict attention to stable categories.  In this setting, the stabilization $\Stab(\cC)$ is
initial amongst presentable stable $\oo$-categories admitting a functor from $\cC$ \cite[1.4.5.5]{HA}, in the sense that if $\cD$ is
a presentable stable $\oo$-category then $\Sigma^\oo_+$ induces an
equivalence   
\[
\Fun^{\L}(\Stab(\cC), \cD) \to \Fun^{\L}(\cC, \cD). 
\]

The $\oo$-category of stable presentable $\oo$-categories $\cP r^{L}_{st}$ is a full subcategory of $\cP r^L$, and
the $\Ind$-category sets up a correspondence between $Cat_{\oo}^{perf}$ with functors that preserve small colimits and compactly generated stable $\oo$-categories with functors which preserve compact objects and colimits denoted $\cP r^{L}_{st, \omega}$.

\subsection{Tensor Products of Stable $\oo$-Categories}
The $\oo$-category $\cP r^{L}_{st}$ of presentable stable $\oo$-categories is a closed symmetric monoidal $\oo$-category with product $\otimes$ and internal
mapping object given by the presentable stable $\oo$-category $\Fun^{\L}(\cA,\cB)$ of colimit-preserving functors \cite[6.3.1.14,6.3.1.17]{HA}.  
Following \cite[4.1.2]{BFN}, we can then define the tensor product on small idempotent-complete stable $\oo$-categories as
\[
\cC \idemtimes \cD = (\Ind(\cC) \otimes \Ind(\cD))^{\omega}.
\]
The tensor product of idempotent-complete small stable $\oo$-categories is characterized by the universal property that maps out of $\cA \otimes \cB$ correspond to maps out of the product $\cA \times \cB$ which preserve finite colimits in each variable \cite[4.4]{BFN}.
If $\cA$ and $\cB$ are arbitrary small stable $\oo$-categories, then we set $\cA \idemtimes \cB:=\Idem(\cA) \idemtimes \Idem(\cB)$.

More precisely, we can define $Cat_{\oo}^{perf}$ as a symmetric monoidal $\oo$-category as follows.  Let $\cP r^{L}_{st, \omega}$ denote the
subcategory of $\cP r^{L}_{st}$ on the compactly-generated stable $\oo$-categories with functors that preserve colimits and compact objects.  The criterion of \cite[2.2.1.2]{HA} implies that $\cP r^{L}_{st, \omega}$ is a symmetric monoidal subcategory of $\cP r^{L}_{st}$; the tensor product of compactly-generated stable $\oo$-categories is itself compactly-generated, as is the unit
$Sp \simeq \Ind(Sp^\omega)$.

For a small stable idempotent-complete $\oo$-category $\cA$ and a presentable $\oo$-category $\cB$, $\Fun^{ex}$ and $\Fun^{L}$ are related by the formula
\[
\Fun^{ex}(\cA, \cB) \simeq \Fun^{L}(\Ind(\cA), \cB),
\]
which follows from \cite[5.3.5.10]{T} and the fact that functors which preserve filtered colimits and finite colimits preserve all colimits. Note that 
\[
\Ind : Cat_{\oo}^{perf} \to \cP r^{L}_{st}
\]
factors through the full subcategory $\cP r^{L}_{st, \omega}$ by definition.  

\begin{prop} \label{perfiscg}
This gives an equivalence of $\oo$-categories
between $Cat_{\oo}^{perf}$ and the subcategory $\cP r^{L}_{st, \omega}$ of
$\cP r^{L}_{st}$ whose objects are the compactly-generated stable $\oo$-categories and whose maps
\[
\Fun^{ex}(\cA,\cB)\simeq\Fun^{\L}_\omega(\Ind(\cA),\Ind(\cB))\subset\Fun^{\L}(\Ind(\cA),\Ind(\cB)), 
\]
are the full subcategory of the colimit-preserving functors $\Ind(\cA) \to \Ind(\cB)$ which preserve compact objects \cite[5.5.7.10]{T}.  
\end{prop}

We regard $Cat_{\oo}^{perf}$ as a symmetric monoidal $\oo$-category via this equivalence.  The observation of \cite[6.3.1.17]{HA} implies that $Cat_{\oo}^{perf}$ is closed.  Hence we have the following result.

\begin{prop}
The $\oo$-category of small idempotent-complete stable $\oo$-categories is a closed symmetric monoidal category with respect to $\idemtimes$.
The unit is the $\oo$-category $Sp^{\omega}$ of compact spectra and the internal mapping object is given for small idempotent-complete stable
$\oo$-categories $\cA$ and $\cB$ by $\Fun^{ex}(\cA, \cB)$.
\end{prop}

\subsection{$\oo$-Operads}\mbox{}\\
In this section we briefly recall the basic definitions and properties of $\oo$-operads, the $\oo$-categorical notion of a colored operad following \cite[Chapter 2]{HA}. Note that ordinary operads are colored operads with only one color. Thus, there is a slight abuse of terminology.  We use the language of $\oo$-operads to define symmetric monoidal $\oo$-categories and also algebra and module objects within a symmetric monoidal $\oo$-category. We then use the language of $\oo$-operads to show the nerve functor is monoidal.

We will need a few technical definitions before we can define an $\oo$-operad.
Let $\Gamma$ denote the category with objects the pointed sets $ \la n \ra = \{*,1,2, \ldots , n\}$ and morphisms those functions which preserve the base point $*$.
\begin{defn}
Let $ \la n \ra^{\circ} = \{1,2, \ldots , n\}$, we will say a morphism $f : \la m \ra \to \la n \ra$ is inert if  for $i \in \la n \ra^{\circ}$, the inverse image $f^{-1}(i)$ has exactly one element.
\end{defn}
Thus, $f : \la m \ra \to \la n \ra$ is inert if $\la n \ra$ is obtained from $\la m \ra$ by identifying some subset of $\la m \ra^{\circ}$ with the base point $*$.

An example we will use is
\begin{ex} 
For $1 \le i \le n$, let $\rho^i : \la n \ra \to \la 1 \ra$  denote the inert morphism
$$\rho^i(j) = \left\{ \begin{array}{rl}
1 & \mbox{if } i=j \\
* & \mbox{otherwise.}
\end{array}\right.$$
\end{ex}

\begin{defn} Let $p: X \to S$ be an inner fibration of simplicial sets (i.e. the fiber over any vertex of $S$ is an $\oo$-category), and let $f:x \to y$ be an edge of $X$, then $f$ is p-coCartesian if the natural map $X_{f/} \rightarrow X_{x/} \times_{S_{p(x)/}} S_{p(f)/}$ is a trivial Kan fibration (or if $f$ satisfies the properties of \cite[2.4.1.8]{T}).  
\end{defn}

\begin{rem}
Informally, if $\bar{f}:s \to s'$ is an edge in $S$ and $f:x \to x'$ lifts $\bar{f}$, then if $f$ is $p$-coCartesian it is determined up to equivalence by $\bar{f}$ and its source $x$.
\end{rem}

\begin{defn} Let $p: X \to S$ be an inner fibration of simplicial sets, $p$ is a coCartesian fibration of simplicial sets if for every edge $\bar{f}:s \to s'$ in $S$, and every vertex $x \in X$ with $p(x)=s$, there exists a $p$-coCartesian edge $f:x \to x'$ with $p(f)=\bar{f}$.
\end{defn}

\begin{defn} 
An $\oo$-operad is an $\oo$-category $\cO^{\otimes}$ and a functor 
$$p: \cO^{\otimes} \to N(\Gamma)$$
satisfying the following conditions \cite[2.1.1.10]{HA}:
\begin{itemize}
\item For every inert morphism $f : \la m \ra \to \la n \ra$ in $\Gamma$ and every object $C \in \cO^{\otimes}_{\la m \ra}$, there is a p-coCartesian morphism $\tilde{f}: C \to C'$ in $\cO^{\otimes}$ lifting $f$.
\item Let $C \in \cO^{\otimes}_{\la m \ra}$ and $C' \in \cO^{\otimes}_{\la n \ra}$ be objects, let $f : \la m \ra \to \la n \ra$ be a morphism in $\Gamma$, and let $map^f_{\cO^{\otimes}}(C,C')$ denote the union of the components of $map_{\cO^{\otimes}}(C,C')$ which lie over $f \in \Hom_{\Gamma}(\la m \ra , \la n \ra)$. Choose $p$-coCartesian morphisms $C' \to C'_i$ lying over the morphism $\rho^i: \la n \ra \to \la 1 \ra$ for each $1 \le i \le n$. The the induced map
$$map^f_{\cO^{\otimes}}(C,C') \to \prod_i map^{f \circ \rho^i}_{\cO^{\otimes}}(C,C'_i)$$
is a homotopy equivalence.
\item For every finite collection $C_1, C_2, \ldots, C_n$ of $\cO^{\otimes}_{\la 1 \ra}$, there exists an object $C$ of $\cO^{\otimes}_{\la n \ra}$ and $p$-coCartesian morphisms $C \to C_i$ covering each $\rho^i$.
\end{itemize}
\end{defn}

\begin{ex} The identity map $N(\Gamma) \to N(\Gamma)$ is an $\oo$-operad. It will be denoted by $Comm^{\otimes}$, and it is the $\oo$-categorical version of the $E_{\oo}$ operad.  More generally, for each $1 \le n \le \oo$, there is a topological category $\tilde{\mathbb{E}}[n]$ \cite[5.1.0.2]{HA} with a natural functor $N(\tilde{\mathbb{E}}[n]) \to N(\Gamma)$, which results in $\oo$-categorical versions of the $\cE_n$ operads.
\end{ex}

\begin{ex} Let $\cO$ be a colored operad, then \cite[2.1.1.7]{HA} constructs a category $\cO^{\otimes}$ whose objects are finite sequences of colors in $\cO$ and a map $\cO^{\otimes} \to\Gamma$. The properties of $\cO^{\otimes}$ imply that the induced map $N(\cO^{\otimes}) \to N(\Gamma)$ is an $\oo$-operad.
\end{ex}

Given an $\oo$-operad $q: \cO^{\otimes} \to N(\Gamma)$ and a coCartesian fibration $p: \cC^{\otimes} \to \cO^{\otimes}$, we will say that $p: \cC^{\otimes} \to \cO^{\otimes}$ is an $\cO$-monoidal $\oo$-category if the composite $q \circ p: \cC^{\otimes} \to N(\Gamma)$ is an $\oo$-operad. Such a map $p$ is called a coCartesian fibration of $\oo$-operads.

\begin{defn} A symmetric monoidal $\oo$-category \cite[2.1.2.18]{HA} is an $\oo$-category $\cC$ and a coCartesian fibration of $\oo$-operads $p: \cC^{\otimes} \to N(\Gamma)$. The underlying $\oo$-category $\cC$ is obtained by $\cC = p^{-1}(\la 1 \ra)$. By abuse of terminology, we say that $\cC$ is a symmetric monoidal $\oo$-category.
\end{defn}

Thus, a symmetric monoidal $\oo$-category is a coCartesian fibration $p: \cC^{\otimes} \to N(\Gamma)$ which induces equivalences of $\oo$-categories $\cC^{\otimes}_{\la n \ra} \simeq \cC^n$, where $\cC$ is the $\oo$-category $\cC^{\otimes}_{\la 1 \ra}$. Moreover the morphisms $\alpha:{\la 0 \ra} \to {\la 1 \ra}$ and $\beta: {\la 2 \ra} \to {\la 1 \ra}$ determine functors
$$\Delta^0 \to \cC \mbox{ and } \cC \times \cC \to \cC$$
which are well-defined up to a contractible space of choice. These maps induce a unit object $\bold{1} \in \cC$ and a monoidal structure $\otimes$ which satisfy all of the properties of a symmetric monoidal category up to homotopy.

\begin{ex} If $\cC$ is a symmetric monoidal category, then $N(\cC)$ is a symmetric monoidal $\oo$-category.
\end{ex}

We now develop the definition of an algebra over an $\oo$-operad. We start with a technical definition.

\begin{defn} 
A morphism $f$ in an $\oo$-operad $p:\cO^{\otimes} \to N(\Gamma)$ is called inert if $p(f)$ is inert and $f$ is p-coCartesian.  
\end{defn}

\begin{defn}
A morphism of $\oo$-operads is a map of simplicial sets $f: \cO^{\otimes} \to \cO'^{\otimes}$ over $N(\Gamma)$ such that $f$ takes inert morphisms in $\cO^{\otimes}$ to inert morphisms in $\cO'^{\otimes}$. The $\oo$-category of $\oo$-operad maps is denoted $Alg_{\cO}(\cO')$ and is considered as a full subcategory of $\Fun_{N(\Gamma)}(\cO^{\otimes},\cO'^{\otimes})$.
\end{defn}

More generally, we have
\begin{defn} If $p: \cC^{\otimes} \to \cO^{\otimes}$ is a fibration of $\oo$-operads (i.e. a categorical fibration \cite[2.2.5.1]{T}) and $f: \cO'^{\otimes} \to \cO^{\otimes}$ is a map of $\oo$-operads, let $Alg_{\cO'/\cO}(\cC)$ be the full subcategory of $\Fun_{\cO}(\cO'^{\otimes},\cC^{\otimes})$ spanned by the maps of $\oo$-operads. When $\cO = \cO'$ and $f$ is the identity, we will denote this $\oo$-category by $Alg_{\cO}(\cC)$, the $\oo$-category of $\cO$-algebra objects in $\cC$.
\end{defn}

\begin{ex} If $\cC$ is a symmetric monoidal $\oo$-category, the $\oo$-category $Alg_{Comm}(\cC^{\otimes})$ is the $\oo$-category of homotopy coherent commutative algebras in $\cC$.
\end{ex}

\begin{defn} Let $p:\cC^{\otimes} \to \cO^{\otimes}$ and $q:\cC^{\otimes} \to \cO^{\otimes}$ be two coCartesian fibrations of $\oo$-operads.  Denote by
$\Fun^{\otimes}_{\cO}(\cC^{\otimes},\cD^{\otimes})$ the subcategory of $\Fun_{\cO}(\cC^{\otimes},\cD^{\otimes})$ spanned by $\oo$-operad maps that furthermore take $p$-coCartesian morphisms to $q$-coCartesian morphisms.  We call this subcategory the $\oo$-category of $\cO$-monoidal functors.
\end{defn}

\begin{ex}  If $\cC$ and $\cD$ are two symmetric monoidal $\oo$-categories, the $\oo$-category $\Fun^{\otimes}_{Comm}(\cC^{\otimes},\cD^{\otimes})$ is the $\oo$-category of symmetric monoidal functors whereas the $\oo$-category $Alg_{\cC/Comm}(\cD^{\otimes})$ is the $\oo$-category of \emph{lax} symmetric monoidal functors. 
\end{ex}

\begin{rem} \label{andrew} Let $\cC$ and $\cD$ be two $\cO$-monoidal $\oo$-categories and $\cF:\cC \to \cD$ be an $\cO$-monoidal functor. Then composition with $\cF$ induces a natural map
$$Alg_{\cO}(\cC) \to Alg_{\cO}(\cD)$$
since an $\cO$-monoidal functor is in particular a map of $\oo$-operads.  Furthermore, if $\cF$ is an equivalence of $\cO$-monoidal $\oo$-categories, it is also an equivalence of $\oo$-operads over $\cO$. 
Thus, it induces an equivalence
$$Alg_{\cO}(\cC) \simeq Alg_{\cO}(\cD)$$
of $\oo$-categories of $\cO$-algebras.
\end{rem}

We now use the language of $\oo$-operads to show the nerve functor is monoidal.  This will be used to construct the functor in our main theorem in the sequel.
\begin{defn} We define a colored operad $\bold{lem}$ \cite[4.2.1]{HA} as follows:
\begin{itemize}
\item The set of objects of $\bold{lem}$ has two objects $\bold{a}$ and $\bold{m}$.
\item Let $\{X_i\}_{i \in I}$ be a finite collection of objects in $\bold{lem}$, and let $Y$ be a another object. If $Y=\bold{a}$, then $Mul_{\bold{lem}}(\{X_i\},Y)$ is the collection of all orderings on $I$ provided that each $X_i = \bold{a}$ and is empty otherwise. If $Y=\bold{m}$, then $Mul_{\bold{lem}}(\{X_i\},Y)$ is the collection of all orderings $\{i_1 < \ldots <i_n \}$ on the set $I$ such that $X_{i_n} = \bold{m}$ and $X_{i_j} = \bold{a}$ for $j < n$.
\item The composition law on $\bold{lem}$ is determined by composing linear orderings in the natural way. 
\end{itemize}
\end{defn}

\begin{rem}
Restricting to the object $\bold{a}$ we get a subcolored operad isomorphic to the associative operad $\bold{Ass}$ of \cite[4.1.1.1]{HA}. Thus, if $\cC$ is a symmetric monoidal category,
$F: \bold{lem} \to \cC$ a map of colored operads, $F(\bold{a}) = A \in \cC$, $F(\bold{m}) = M \in \cC$, then the unique  $\phi \in Mul_{\bold{lem}}(\{\bold{a},\bold{m} \}, \bold{m})$ determines a map $\phi: A \otimes M \to M$, which exhibits $M$ as a left $A$-module, where $A$ is an associative algebra object in $\cC$.
\end{rem}

\begin{defn} Applying \cite[2.1.1.7]{HA} to the colored operad $\bold{lem}$, we obtain a category $\bold{lem}^{\otimes}$ and a map $\bold{lem}^{\otimes} \to \Gamma$. Thus, the nerve $\cL \cM^{\otimes} := N(\bold{lem}^{\otimes}) \to N(\Gamma)$ is the $\oo$-operad which controls left module objects over an associative algebra. There is also an $\oo$-operad $\cA ss^{\otimes}$ obtained by applying the same construction to $\bold{Ass}$.
\end{defn}

\begin{defn}  Let $\cC^{\otimes} \to \cA ss^{\otimes}$ be a fibration of $\oo$-operads, and $\cM$ an $\oo$-category. A weak enrichment of $\cM$ over $\cC^{\otimes}$ is a fibration of $\oo$-operads $q: \cO^{\otimes} \to \cL \cM^{\otimes}$ such that $\cO^{\otimes}_{\bold{a}} \simeq \cC^{\otimes}$ and $\cO^{\otimes}_{\bold{m}} \simeq \cM$. We let $LMod(\cM) := Alg_{\cL \cM}(\cO)$ be the $\oo$-category of left module objects in the $\oo$-category $\cM$ \cite[4.2.1.13]{HA}.
\end{defn}

We will need the following result we will need in our comparison theorem.
\begin{prop} \label{nerve} Let $\cC$ be a monoidal category, and $s: \bold{lem} \to \cC$ a left module object of $\cC$, then $N(s): \cL \cM^{\otimes} \to N(\cC)^{\otimes}$ defines a left module object of $N(\cC)$.
\end{prop}
\begin{proof} Let $N(\cC)^{\otimes} \to \cA ss^{\otimes}$ be a monoidal $\oo$-category. Then 
$$\cO^{\otimes} = N(\cC)^{\otimes} \times_{\cA ss^{\otimes}} \cL \cM^{\otimes}$$
exhibits the $\oo$-category $N(\cC)$ as weakly enriched over $N(\cC)^{\otimes}$. The map $s: \bold{lem} \to \cC$ of colored operads induces a map $s: \bold{lem}^{\otimes} \to \cC^{\otimes}$ of categories, which in turn induces a map $N(s): \cL \cM^{\otimes} \to N(\cC)^{\otimes}$ of $\oo$-operads. This gives an element of $N(s) \in Fun_{\cL \cM}(\cL \cM, \cO)$, and furthermore, $N(s)$ sends inert morphisms to inert morphism, and hence, $N(s) \in Alg_{\cL \cM}(\cO)$ determines a left module object in $N(\cC)$.
\end{proof}

\subsection{From Cofibrant Spectral Categories to Flat Spectral Categories}
In this section we will denote a symmetric monoidal $\oo$-category $\cA^{\otimes}$ by additionally having the superscript ``$\otimes$".  

When $C$ is a symmetric monoidal model category such that the weak equivalences are preserved by the product, the underlying $\oo$-category of $C$ produces a symmetric monoidal $\oo$-category $N(C^c)[W^{-1}]^{\otimes}$ by \cite[4.1.3.4]{HA}, \cite[4.1.3.6]{HA}.  For instance, when $C$ is a symmetric monoidal model category, the cofibrant objects $C^c$ form a symmetric monoidal category with weak equivalences preserved by the product, and this symmetric monoidal category induces a symmetric monoidal structure on $N(C^c)[W^{-1}]$, which we will denote by $N(C^c)[W^{-1}]^{\otimes}$.

The category $Cat_{\cS}$ has a closed symmetric monoidal structure given by the pointwise smash product.  However, the pointwise smash product of cofibrant spectral categories is not necessarily cofibrant as is also the case for dg categories. Consequently, the model structure on $Cat_{\cS}$ is not monoidal for the pointwise smash structure.  

To resolve this problem, we will use the notion of flat objects and functors following \cite[Chapter 3]{BGT2}. 

\begin{defn}
A functor between model categories is flat if it preserves weak equivalences and colimits. An object X of a model category (whose underlying category is monoidal) is then said to be flat if the functor $X \otimes (-)$ is a flat functor. 
\end{defn}

For instance, cofibrant spectra are in particular flat. The following definition given in \cite[Chapter 3]{BGT2} satisfies this notion of flatness. 

\begin{defn} \label{flatsc} A spectral category C is pointwise-cofibrant if each morphism spectrum $C(x, y)$ is a cofibrant spectrum. 
\end{defn} 

The following proposition summarizes the properties of pointwise-cofibrant spectral categories that we will need \cite[3.2]{BGT2}.
\begin{prop}
The following are properties of pointwise-cofibrant spectral categories:
\begin{itemize}
\item Every spectral category is functorially Morita equivalent to a pointwise-cofibrant spectral category with the same objects.
\item The subcategory of pointwise-cofibrant spectral categories is closed under the pointwise smash product.
\item A pointwise-cofibrant spectral category is flat with respect to the pointwise smash product of spectral categories.
\item If $C$ and $D$ are pointwise-cofibrant spectral categories, the pointwise smash product $C \wedge D$ computes the derived smash product $C \wedge^L D$.
\end{itemize}
\end{prop}

Therefore, we use the subcategory $Cat^{flat}_{\cS}$ of pointwise-cofibrant spectral categories to produce a suitable symmetric monoidal model of the $\oo$-category of idempotent-complete small stable $\oo$-categories. For instance, we have that \cite[3.4]{BGT2}:

\begin{prop}
The functor induced by cofibrant replacement $Cat^{flat}_{\cS} \to Cat^c_{\cS}$, induces a categorical equivalence $N(Cat^{flat}_{\cS})[W^{-1}] \simeq N(Cat^c_{\cS})[W^{-1}]$, where $W$ is the class of Morita equivalences.
\end{prop}

\begin{cor} We have an equivalence $N(Cat^{flat}_{\cS})[W^{-1}]  \simeq Cat^{perf}_{\oo}$.
\end{cor}

Moreover, combined with \cite[4.1.3.4]{HA} and \cite[4.1.3.6]{HA}, we see that:
\begin{prop} The $\oo$-category $N(Cat^{flat}_{\cS})[W^{-1}]$ can be promoted to a symmetric monoidal $\oo$-category $N(Cat^{flat}_{\cS})[W^{-1}]^{\otimes}$.
\end{prop} 

Then, \cite[3.5]{BGT2} also proves:
\begin{thm} There is an equivalence of symmetric monoidal $\oo$-categories 
$$N(Cat^{flat}_{\cS})[W^{-1}]^{\otimes} \simeq (Cat^{perf}_{\oo})^{\otimes}.$$
\end{thm}

Let $R$ be an $E_{\oo}$-ring spectrum. There is a stable presentable $\oo$-category $R \dash \Mod$ of modules over $R$ in the $\oo$-category of spectra \cite[7.1.1.5]{HA}.

\begin{defn} The $\oo$-category of perfect modules over $R$, $\Perf(R)$, is the smallest stable subcategory of the $\oo$-category $R \dash \Mod$ which contains $R$ and is closed under retracts \cite[7.2.5.1]{HA}.
\end{defn}

\begin{prop}
Let $R$ be and $E_{\oo}$-ring spectrum, then the $\oo$-category $\Perf(R)$ is a commutative algebra object in $(Cat^{perf}_{\oo})^{\otimes}$.
\end{prop}

We see that the notion of module categories over $\Perf(R)$ is well-defined.  In particular, by \cite[4.2.3.7]{HA}:
\begin{cor}
The $\oo$-category $\Mod_{\Perf(R)}((Cat^{perf}_{\oo})^{\otimes})$ is a presentable $\oo$-category. 
\end{cor}

\subsection{The Barr-Beck-Lurie Theorem}\mbox{}\\
Given an adjunction of ordinary categories $F: C\rightleftarrows D: G$, the classical Barr-Beck theorem gives an equivalence between $T \dash algebras$ and $D$ where $T$ is the monad $T= G \circ F$, provided that $G$ is conservative and preserves certain colimits.  The Barr-Beck theorem is a useful tool for identifying categories as categories of modules.

In \cite[4.7]{HA}, Lurie provides an $\oo$-categorical version of the Barr-Beck theorem.  Namely, if $F: C \rightleftarrows D: G$ is an adjunction of $\oo$-categories, then there is an equivalence between $T \dash algebras$  and $D$ as long as $G$ is conservative and preserves certain colimits (i.e. so-called G-split simplicial objects).  

A corollary, \cite[4.7.4.16]{HA}, of the $\oo$-categorical Barr-Beck theorem will provide the machinery for our main theorem, we record it here for use in the next section.
\begin{thm}
Suppose we are given a commutative diagram of $\oo$-categories
\[\xymatrix{
C  \ar[dr]_{G} \ar[rr]^{U} && C' \ar[dl]^{G'}  \\
  & D &
}\]

Assume that:
\begin{enumerate}
\item The functors $G$ and $G'$ admit left adjoints $F$ and $F'$.
\item $C$ admits geometric realizations of simplicial objects, which are preserved by $G$.
\item $C'$ admits geometric realizations of simplicial objects, which are preserved by $G'$.
\item The functors $G$ and $G'$ are conservative.
\item For each object $d \in D$, the unit map $d \to GF(d) \simeq G'(UF(d))$ induces an equivalence $G'F'(d) \to GF(d)$ in $D$.
\end{enumerate}
Then $U$ is an equivalence of $\oo$-categories.
\end{thm}

\section{The Main Theorem}\mbox{}\\
We now prove the main theorem of this paper.  We show that if $R$ is an $E_{\oo}$-ring spectrum, then the $\oo$-category of module categories over $\Perf(R)$ in $Cat^{perf}_{\oo}$ corresponds to the underlying $\oo$-category of categories enriched in $R$-module spectra with weak equivalences as Morita equivalences.  Note that in particular, $Hk$ is an $E_{\oo}$-ring spectrum, and this will provide the link between dg categories over $k$ and $k$-linear stable $\oo$-categories.  We will discuss this more in the next section.  We first state and prove our theorem.  

Let $W$ be the class of Morita equivalences in $Cat_{\cS}$ and let $W'$ be the class of Morita equivalences in $\Mod_{\Perf(R)^{\cell}}(Cat_{\cS})$. Recall these Morita equivalences are determined by the forgetful functor to $Cat_{\cS}$. The natural monoidal functor 
$$\theta: N(Cat_{\cS}^{flat})^{\otimes} \to N(Cat_{\cS}^{flat})[W^{-1}]^{\otimes}$$ 
defined by the localization of symmetric monoidal $\oo$-categories (which exists by \cite[4.1.3.4]{HA}) determines a natural map 
$$\theta: N(\Mod_{\Perf(R)^{\cell}}(Cat_{\cS})^{flat})[W'^{-1}] \rightarrow \Mod_{\Perf(R)}(N(Cat_{\cS}^{flat})[W^{-1}]^{\otimes}).$$ 
This map $\theta$ is well-defined because the nerve of a module category over $\Perf(R)^{\cell}$ in $Cat_{\cS}$ is naturally a module category over the $\oo$-category $\Perf(R)$. This follows by \ref{nerve} and by the observations that the $\oo$-category associated to $\Perf(R)^{\cell}$ is $\Perf(R) \in Cat^{perf}_{\oo}$ by \ref{cell1}.

\begin{thm} \label{main} The functor 
$$\theta: N(\Mod_{\Perf(R)^{\cell}}(Cat_{\cS})^{flat})[W'^{-1}] \rightarrow \Mod_{\Perf(R)}(N(Cat_{\cS}^{flat})[W^{-1}]^{\otimes})$$
is an equivalence of presentable $\oo$-categories.  
\end{thm}

\begin{proof}
As discussed in the previous section, the proof of this theorem is an application of the Barr-Beck-Lurie theorem.  Namely, we show that this diagram satisfies the hypotheses of \cite[4.7.4.16]{HA}:

Consider the diagram 
\xymatrixcolsep{-0.35in}
\[\xymatrix{
N(\Mod_{\Perf(R)^{\cell}}(Cat_{\cS})^{flat})[W'^{-1}]  \ar[dr]_{G} \ar[rr]^{\theta} && \Mod_{\Perf(R)}(N(Cat_{\cS}^{flat})[W^{-1}]^{\otimes}) \ar[dl]^{G'}  \\
  & N(Cat_{\cS}^{flat})[W^{-1}]  &
}\]

\begin{enumerate}
\item The two $\oo$-categories given by the source and target of $\theta$ admit geometric realizations of simplicial objects. In fact, both of these categories are presentable $\oo$-categories.  The former is a presentable $\oo$-category by \ref{presentable}. The latter is a presentable $\oo$-category by \cite[4.2.3.7]{HA}, that is, $\oo$-categories of modules over a presentable $\oo$-category is a presentable $\oo$-category.  

\item The functors $G$ and $G'$ admit left adjoints $F$ and $F'$.  The existence of a left adjoint to G follows from the fact that $G$ preserves limits. That is, limits in  $N(\Mod_{\Perf(R)^{\cell}}(Cat_{\cS})^{flat})[W'^{-1}]$ are computed in $N(Cat_{\cS}^{flat})[W^{-1}]$. The existence of a left adjoint to $G'$ follows from \cite[4.2.4.8]{HA}. 

\item The functor $G'$ is conservative and preserves geometric realizations.  The former is true by \cite[4.2.3.2]{HA}. The latter is true by \cite[4.2.3.5]{HA}. 

\item The functor $G$ is conservative and preserves geometric realizations. The first assertion is immediate from the definition of the weak equivalences in 
$\Mod_{\Perf(R)^{\cell}}(Cat_{\cS}^{flat})$. For the second assertion, it suffices to show that the left action of $\Perf(R)^{\cell}$ on $Cat_{\cS}^{flat}$ preserves colimits. But this follows from the definition of the pointwise smash product of spectral categories.

\item The morphism $G' \circ F' \to G \circ F$ is an equivalence.  This follows because the natural map $\cC \to \Perf(R) \otimes \cC$ induces and equivalence $G'F'(\cC) \cong \Perf(R) \otimes \cC$ by \cite[4.2.4.8]{HA}.
\end{enumerate}
Therefore, the functor is an equivalence.
\end{proof}

Using the equivalence
$$N(Cat_{\cS}^{flat})[W^{-1}]^{\otimes} \simeq (Cat^{perf}_{\oo})^{\otimes}$$
we have:

\begin{cor} There is an equivalence of $\oo$-categories
$$N(\Mod_{\Perf(R)^{\cell}}((Cat_{\cS})^{flat})[W'^{-1}] \simeq \Mod_{\Perf(R)}((Cat^{perf}_{\oo})^{\otimes}).$$
\end{cor}

Using our main theorem above and the equivalence of $\oo$-categories \ref{newcomparison}
$$N(Cat_{R \dash \Mod})[W^{-1}] \simeq N(\Mod_{\Perf(R)^{\cell}}(Cat_{\cS}))[W'^{-1}].$$
we deduce the following corollary.

\begin{cor} \label{cor1}
There is an equivalence of $\oo$-categories
$$ N(Cat_{R \dash \Mod})[W^{-1}] \simeq \Mod_{\Perf(R)}((Cat^{perf}_{\oo})^{\otimes}).$$
\end{cor}

Applying this theorem to the Eilenberg-MacLane ring spectrum $Hk$ yields the following corollary.

\begin{cor} \label{cor2}
There is an equivalence of $\oo$-categories
$$ N(Cat_{Hk \dash \Mod})[W^{-1}] \simeq \Mod_{\Perf(Hk)}((Cat^{perf}_{\oo})^{\otimes}).$$
\end{cor}

We now combine this result with the Quillen equivalence \cite{Tab1}
$$Cat_{Hk \dash \Mod} \simeq Cat^k_{dg},$$ 
to obtain our main corollary.

\begin{cor} \label{cor3}
There is an equivalence of $\oo$-categories
$$ N(Cat^k_{dg})[W^{-1}] \simeq \Mod_{\Perf(Hk)}((Cat^{perf}_{\oo})^{\otimes}).$$ In other words, the underlying $\oo$-category of the Morita model category structure on $Cat^k_{dg}$ is equivalent to the $\oo$-category of idempotent-complete $k$-linear small stable $\oo$-categories. 
\end{cor}

Applying the $\Ind$ functor, we have a corresponding statement for large $k$-linear stable $\oo$-categories.

\begin{prop} \label{smalltolarge}
There is an equivalence of $\oo$-categories
$$\Ind: \Mod_{\Perf(Hk)}((Cat^{perf}_{\oo})^{\otimes}) \simeq \Mod_{Hk \dash \Mod}((\cP r^{L}_{st, \omega})^{\otimes})$$
\end{prop}

\begin{proof} 
First, $Cat^{perf}_{\oo}$ and $\cP r^{L}_{st, \omega}$ are symmetric monoidal $\oo$-categories and the $\Ind$-construction is symmetric monoidal by \cite[4.4]{BFN}.   
That is, there is a commutative diagram
\[\xymatrix{
(Cat^{perf}_{\oo})^{\otimes} \ar[dr] \ar[rr]^{\Ind}_{\simeq} & & (\cP r^{L}_{st, \omega})^{\otimes} \ar[dl] \\
 & N(\Gamma)^{\otimes}
}\]

Second, $Ass$-algebra objects and left module objects can be defined in any symmetric monoidal $\oo$-category. Since the functor $\Ind$ induces an equivalence of symmetric monoidal $\oo$-categories, there are induced equivalences
\[\xymatrix{
Alg_{Ass}((Cat^{perf}_{\oo})^{\otimes} \ar[r]^{\Ind}_{\simeq} & Alg_{Ass}((\cP r^{L}_{st, \omega})^{\otimes})
}\]
and
\[\xymatrix{
Alg_{\cL \cM}((Cat^{perf}_{\oo})^{\otimes} \ar[r]^{\Ind}_{\simeq} & Alg_{\cL \cM}((\cP r^{L}_{st, \omega})^{\otimes})
}\]
by \ref{andrew}.

Finally, the $\oo$-category $\Mod_{\Perf(Hk)}((Cat^{perf}_{\oo})^{\otimes})$ is realized as the fiber product
$$Alg_{\cL \cM}((Cat^{perf}_{\oo})^{\otimes}) \times_{Alg_{Ass}((Cat^{perf}_{\oo})^{\otimes})} \Perf(Hk).$$

Similarly, the $\oo$-category $\Mod_{Hk \dash \Mod}((\cP r^{L}_{st, \omega})^{\otimes})$ is realized as the fiber product
$$Alg_{\cL \cM}((\cP r^{L}_{st, \omega})^{\otimes}) \times_{Alg_{Ass}((\cP r^{L}_{st, \omega})^{\otimes})} Hk \dash \Mod.$$

Thus, since $\Ind(\Perf(Hk)) \simeq Hk \dash \Mod$, the $\Ind$ functor induces a well-defined equivalence
\[\xymatrix{
Alg_{\cL \cM}((Cat^{perf}_{\oo})^{\otimes}) \times_{Alg_{Ass}((Cat^{perf}_{\oo})^{\otimes})} \Perf(Hk) \ar[r]^{\Ind}_{\simeq} &Alg_{\cL \cM}((\cP r^{L}_{st, \omega})^{\otimes}) \times_{Alg_{Ass}((\cP r^{L}_{st, \omega})^{\otimes})} Hk \dash \Mod
}\]

The result follows.
\end{proof}

We now have our second main corollary.
\begin{cor} \label{cor4}
There is an equivalence of $\oo$-categories
$$ N(Cat^k_{dg})[W^{-1}] \simeq \Mod_{Hk \dash \Mod}((\cP r^{L}_{st, \omega})^{\otimes}).$$ In other words, the underlying $\oo$-category of the Morita model category structure on $Cat^k_{dg}$ is equivalent to the $\oo$-category of compactly-generated presentable $k$-linear stable $\oo$-categories with functors that preserve compact objects and colimits. 
\end{cor}



\begin{thebibliography}{1}

\bibitem[Ba]{Ba} Barwick, Clark. On left and right model categories and left and right Bousfield localizations. Homology, Homotopy Appl. 12 (2) (2010), 245Ð320.

\bibitem[BK]{BK} Bondal, Alexei and Kapranov, Mikhail. Enhanced triangulated categories. Math. USSRSbornik, vol. 70 (1991), no. 1, 93Ð107. (Russian original: vol. 181 (1990), no. 5)

\bibitem[BFN]{BFN}  Ben-Zvi, David, Francis, John, and Nadler, David. Integral transforms and Drinfeld centers in derived algebraic geometry.  J. Amer. Math. Soc. 23 (2010), 909-966.

\bibitem[BGT1]{BGT1} Blumberg, Andrew, Gepner, David, and Tabuada, Goncalo. A Universal Characterization of Higher Algebraic K-theory. 
Geometry And Topology 14 (2010) 1165Ð1242.  \href{http://arxiv.org/pdf/1001.2282v4.pdf}{link}

\bibitem[BGT2]{BGT2} Blumberg, Andrew, Gepner, David, and Tabuada, Goncalo. Uniqueness of the Multiplicative Cyclotomic Trace. \href{http://arxiv.org/pdf/1103.3923v2.pdf}{link}

\bibitem[BM]{BM} Blumberg, Andrew and Mandell, Mike. Localization theorems in topological Hochschild homology and topological cyclic homology. Geom. and Top. 16 (2012), 1053-1120.

\bibitem[CT]{CT} Cisniski, Denis-Charles, Tabuada, Goncalo. Symmetric Monoidal Structure on Non-commutative Motives.  Journal of K-theory. Volume 9 (2012), No. 2, 201-268.

\bibitem[Drinfeld]{Drinfeld} Drinfeld, Vladimir. DG quotients of DG categories. J. Algebra 272 (2) (2004), 643Ð691.

\bibitem[Du]{Du} Dubuc, Eduardo. Kan Extensions in Enriched Category Theory. Lecture Notes in Math., vol. 145, Springer-Verlag, 1967, 1-77.

\bibitem[EKMM]{EKMM} Elmendorf, Anthony, Kriz, Igor, Mandell, Michael and May, Peter. Rings, Modules, and Algebras in Stable Homotopy Theory. Volume 47 of Mathematical Surveys and Monographs. American Mathematical Society, Providence, RI, 1997.

\bibitem[G]{G} Gaitsgory, Dennis.  Generalities on DG Categories \href{http://www.math.harvard.edu/~gaitsgde/GL/textDG.pdf}{link}.

\bibitem[Hovey]{Hovey} Hovey, Mark. Model Categories Mathematical Surveys and Monographs, Volume 63, AMS (1999) 

\bibitem[HSS]{HSS} Hovey, Mark, Shipley, Brooke, and Smith, Jeff. Symmetric spectra. J. Amer. Math. Soc. 13 (2000), no.1, 149Ð208.

\bibitem[K]{K} Keller, Bernhard. On differential graded categories International Congress of Mathematicians. Vol. II, 151Ð190, Eur. Math. Soc., ZŸrich, 2006.

\bibitem[T]{T} Lurie, Jacob. Higher Topos Theory. Annals of Mathematics Studies, 170, Princeton University Press, 2009. \href{http://www.math.harvard.edu/~lurie/papers/croppedtopoi.pdf}{link}

\bibitem[HA]{HA} Lurie, Jacob. Higher Algebra. Preprint, 2011. \href{http://www.math.harvard.edu/~lurie/papers/higheralgebra.pdf}{link}.

\bibitem[L3]{L3}  Lurie, Jacob. What is an $\oo$-category?, Notices of the AMS, 55 (8), September, 2008

\bibitem[Sch]{Sch} Schwede, Stefan. S-modules and Symmetric Spectra. Math. Ann. 319, 517Ð532 (2001)

\bibitem[SS1]{SS1} Schwede, Stefan and Shipley, Brook. Algebras and Modules in Monoidal Model Categories,  Proc. London Math. Soc. 80 (2000), 491-511, \href{http://homepages.math.uic.edu/~bshipley/monoidal.pdf}{link}.

\bibitem[SS2]{SS2} Schwede, Stefan and Shipley, Brook. Stable Model Categories are Categories of Modules. Topology 42 (2003), 103-153.

\bibitem[TabS]{TabS} Tabuada, Gon\c{c}alo. Homotopy theory of spectral categories. Adv. in Math. 221(4) (2009), 1122Ð1143.

\bibitem[TabM]{TabM}Tabuada, Gon\c{c}alo, Invariants additifs de dg-catŽgories. Internat. Math. Res. Notices 53 (2005), 3309Ð3339.

\bibitem[Tab1]{Tab1} Tabuada, Gon\c{c}alo.  Topological Hochschild and Cyclic Homology for Differential Graded Categories. \href{http://arxiv.org/pdf/0804.2791v2.pdf}{link}

\bibitem[Tab2]{Tab2} Tabuada, Gon\c{c}alo. Une Structure de Categorie de Modeles de Quillen sur la Categorie des DG-categories. C.R. Math. Acad. Sci. Paris (2005), 340(1) 15-19 . \href{http://arxiv.org/pdf/math/0407338.pdf}{link}

\bibitem[To1]{To1} T\"{o}en, Bertrand. The homotopy theory of dg-categories and derived Morita Theory. Inventiones mathematicae, Volume 167, Number 3, 2007 , 615-667(53).

\bibitem[To2]{To2} T\"{o}en, Bertrand. Lectures on dg-categories.  \href{http://www.math.univ-toulouse.fr/~T\"{o}en/swisk.pdf}{link}

\end{thebibliography}
\end{document}